\documentclass[11pt]{amsart}

\usepackage{amsmath, amssymb,amsthm,color}
\usepackage[all]{xy}
\usepackage{tikz-cd}
\usetikzlibrary{arrows.meta,decorations.markings,arrows}

\newcommand{\nc}{\newcommand}
\newcommand{\cC}{\mathtt C}
\newcommand{\cD}{\mathtt D}

\newcommand{\uz}{{\underline{z}}}

\newcommand{\ul}{{\underline{\lambda}}}

\newcommand{\Hom}{\operatorname{Hom}}
\newcommand{\C}{\mathbb{C}}
\newcommand{\R}{\mathbb R}
\newcommand{\PP}{\mathbb P}

\newcommand{\cE}{\mathcal E}

\newcommand{\wt}{\widetilde}

\newcommand{\Cx}{\C^\times}
\newcommand{\Rx}{\R^\times}
\nc{\vareps}{\varepsilon}
\nc{\arrg}{\operatorname{arg}}
\nc{\Cf}{\mathcal{T}}

\nc{\circQ}{Q^\circ}
\nc{\circCQ}{\CQ^\circ}
\nc{\circcircQ}{Q^{\circ \circ}}
\nc{\circcircCQ}{\CQ^{\circ \circ}}
\nc{\circY}{\ft}
\nc{\Ycirc}{\overline \ft^\circ}
\nc{\Cfcirc}{\overline \Cf^\circ}
\nc{\circCY}{\Cf}
\nc{\circcircY}{\ft^{\circ}}
\nc{\circcircCY}{\Cf^{\circ}}
\nc{\bU}{\tilde{U}}
\nc{\bCU}{\tilde{\CU}}
\nc{\bCV}{\tilde{\CV}}
\nc{\bbV}{V}
\nc{\bbU}{U}
\nc{\bbCU}{{\CU}}
\nc{\bbCV}{{\CV}}
\nc{\circcircM}{M^{\circ \circ}}

\nc{\Gr}{\mathbb{G}}

\nc{\A}{{\mathcal A}}
\nc{\ol}{\overline}
\nc\tboxtimes{\wt{\boxtimes}}
\nc{\alp}{\alpha}
\nc{\Wh}{\operatorname{Wh}}
\nc{\IC}{{\mathcal{IC}}}
\nc{\Uhg}{{U_\hbar \fg}}
\nc{\tW}{{\tilde{\Gr}}}
\nc{\bM}{\mathbf M}
\nc{\la}{\lambda}
\nc{\tM}{\widetilde{M}}
\nc{\tCM}{\widetilde{\mathcal M}}

\nc{\Fr}[1]{\overline F_{#1}(\BR)}
\nc{\Mr}[1]{\overline M_{#1}(\BR)}
\nc{\Ms}[1]{\overline M_{#1}^{split}}
\renewcommand{\Mc}[1]{\overline M_{#1}^{comp}}
\nc{\Fs}[1]{\overline \CF_{#1}^{split}}
\nc{\Fc}[1]{\overline \CF_{#1}^{comp}}

\renewcommand{\l}{\lambda}

\nc{\BA}{{\mathbb{A}}} \nc{\BC}{{\mathbb{C}}}
\nc{\BQ}{{\mathbb{Q}}}
\nc{\BL}{{\mathbb{L}}} \nc{\BN}{{\mathbb{N}}}
\nc{\BP}{{\mathbb{P}}} \nc{\BR}{{\mathbb{R}}}
\nc{\BZ}{{\mathbb{Z}}} 
\nc{\BG}{{\mathbb{G}}}

\nc{\CA}{{\mathcal{A}}} \nc{\CB}{{\mathcal{B}}}
\nc{\CD}{{\mathcal{D}}}
\nc{\CE}{{\mathcal{E}}} \nc{\CF}{{\mathcal{F}}}
\nc{\CG}{{\mathcal{G}}} \nc{\CH}{{\mathcal{H}}}
\nc{\CI}{{\mathcal{I}}}  \nc{\CJ}{{\mathcal{J}}}
\nc{\CL}{{\mathcal{L}}}
\nc{\CM}{{\mathcal{M}}} \nc{\CN}{{\mathcal{N}}}
\nc{\CO}{{\mathcal{O}}} \nc{\CP}{{\mathcal{P}}}
\nc{\CQ}{{\mathcal{Q}}} \nc{\CR}{{\mathcal{R}}}
\nc{\CS}{{\mathcal{S}}} \nc{\CT}{{\mathcal{T}}}
\nc{\CU}{{\mathcal{U}}} \nc{\CV}{{\mathcal{V}}}
\nc{\CX}{{\mathcal{X}}}
\nc{\CW}{{\mathcal{W}}} \nc{\CZ}{{\mathcal{Z}}}

\nc{\ff}{{\mathfrak{f}}} \nc{\fv}{{\mathfrak{v}}}
\nc{\fa}{{\mathfrak{a}}} \nc{\fb}{{\mathfrak{b}}}
\nc{\fd}{{\mathfrak{d}}} \nc{\fe}{{\mathfrak{e}}}
\nc{\fg}{{\mathfrak{g}}} \nc{\fgl}{{\mathfrak{gl}}}
\nc{\fh}{{\mathfrak{h}}} \nc{\fri}{{\mathfrak{i}}}
\nc{\fj}{{\mathfrak{j}}} \nc{\fk}{{\mathfrak{k}}}
\nc{\fm}{{\mathfrak{m}}} \nc{\fn}{{\mathfrak{n}}}
\nc{\ft}{{\mathfrak{t}}} \nc{\fu}{{\mathfrak{u}}}
\nc{\fw}{{\mathfrak{w}}} \nc{\fz}{{\mathfrak{z}}}
\nc{\fl}{{\mathfrak{l}}}
\nc{\fp}{{\mathfrak{p}}} \nc{\frr}{{\mathfrak{r}}}
\nc{\fs}{{\mathfrak{s}}} \nc{\fsl}{{\mathfrak{sl}}}
\nc{\hsl}{{\widehat{\mathfrak{sl}}}}
\nc{\hgl}{{\widehat{\mathfrak{gl}}}}
\nc{\hg}{{\widehat{\mathfrak{g}}}}
\nc{\chg}{{\widehat{\mathfrak{g}}}{}^\vee}
\nc{\hn}{{\widehat{\mathfrak{n}}}}
\nc{\chn}{{\widehat{\mathfrak{n}}}{}^\vee}
\nc{\Xiset}{\Xi\text{-}\mathtt{Set}}
\nc{\Lset}{\Lambda_+\text{-}\mathtt{Set}}
\nc{\set}{\mathtt{Set}}
\nc{\XiCCob}{\Xi\text{-}\mathtt{CCCat}}
\nc{\XiCov}{\Xi\text{-}\mathtt{Cov}}
\nc{\gcrys}{\fg\text{-}\mathtt{Crys}}
\nc{\sqcu}{\bigsqcup\limits}

\usepackage{accents}

\newtheorem{cor}{Corollary}[section]

\newtheorem{lem}[cor]{Lemma}
\newtheorem{prop}[cor]{Proposition}

\newtheorem{conj}[cor]{Conjecture}
\newtheorem{thm}[cor]{Theorem}

\theoremstyle{definition}
\newtheorem{defn}[cor]{Definition}
\newtheorem{rem}[cor]{Remark}

\theoremstyle{remark}

\begin{document}

\title{Cactus flower spaces and monodromy of Bethe vectors}
\author{Joel Kamnitzer and Leonid Rybnikov}

\begin{abstract}
    We continue the study of cactus flower moduli spaces $\overline{F}_n$ and Gaudin models started in \cite{IKLPR,IKR}. We show that isomorphism classes of operadic coverings of the real form $\overline{F}_n(\mathbb{R})$ are naturally one-to-one with equivalence classes of concrete coboundary monoidal categories (i.e. coboundary monoidal categories that admit a faithful monoidal functor to sets) with certain semisimplicity and finiteness conditions. Following the strategy of \cite{HKRW}, for any complex semisimple Lie algebra $\fg$, we recover Kashiwara $\mathfrak{g}$-crystals, as a concrete coboundary category, from the coverings given by Bethe eigenlines for inhomogeneous Gaudin models. Using this, we compute the monodromy of Bethe eigenlines for trigonometric Gaudin models over two different real loci. In the particular case of minuscule highest weights, this can be regarded as combinatorial version of the wall-crossing conjecture of Bezrukavnikov and Okounkov for quantum cohomology of symplectic resolutions in the case of minuscule resolutions of slices in the affine Grassmannian.
\end{abstract}

\maketitle

\section{Introduction}

\subsection{Cactus flower spaces} Let $\overline{M}_{n+1}$ denote the Deligne-Mumford space of stable rational curves with $n+1$ marked points. The points of $\overline{M}_{n+1}$ are isomorphism classes of curves of genus $0$, with $n+1$ ordered marked points and possibly with simple nodes, such that each component has at least $3$ distinguished points (either marked points or nodes). Since one can always assume that the $(n+1)$-th marked point has the coordinate $\infty$, this space can be regarded as a compactification of configurations of $n$ distinct points on the complex affine line, modulo simultaneous translations and dilations. Informally, the topology of $\overline{M}_{n+1}$ is determined by the following rule: when some of the distinguished points (marked or nodes) from the same component collide, they bubble off into a new component, and this procedure can be iterated while there remain components with at least $4$ distinguished points. 

In \cite{IKLPR}, along with Ilin, Li, and Przytycki, we introduced and studied the \emph{cactus flower space} $\overline{F}_n$, i.e. the moduli space of \emph{cactus flower curves}, which compactifies configurations of $n$ distinct points on the complex affine line, modulo translations (but not dilations). This makes $\infty$ an exceptional marked point, so its collision with other marked points behaves differently from the collision of usual marked points with each other. Namely, any group of points approaching infinity while keeping the mutual distances bounded forms a new component attached to the original one at the infinity point -- thus forming a bouquet-like structure at $\infty$. It can be regarded as an additive analog of the Deligne--Mumford compactification $\overline{M}_{n+2}$ (that is naturally a compactification of the space of configurations of $n$ points on $\mathbb{C}^\times$ modulo dilations), and, moreover, there is a one-parameter family that deforms $\overline{F}_n$ to $\overline{M}_{n+2}$: there is a variety $\overline{\mathcal{F}}_n$ over $\mathbb{A}^1$ such that its fiber over $\varepsilon\ne0\in\mathbb{A}^1$ is $\overline{M}_{n+2}$ and the special fiber at $0\in\mathbb{A}^1$ is $\overline{F}_n$. 

The real form $\overline{M}_{n+1}(\mathbb{R})$ is a combinatorial space glued from associahedrons (cf. \cite{Kap}).  Dually, Davis, Januszkiewicz, and Scott \cite{DJS} presented it as a cube complex, describing it as a $K(\pi,1)$ space for the \emph{cactus group} $C_n$. Similarly, the real locus of the cactus flower space, $ \overline F_n(\BR) $, is a combinatorial space that admits a nice presentation as a cube complex, see \cite{IKLPR}. Moreover, the deformation space $\overline{\mathcal{F}}_n$ has two particularly interesting real forms, both deforming the real form $\overline{F}_n(\mathbb{R})$ over the real affine line. First, \emph{the split real form} $\Fs{n}$ whose fiber at $\varepsilon\ne0$ is the usual real form $\Ms{n+2}=\overline{M}_{n+2}(\mathbb{R})$, the closure of the configurations of real points on $\mathbb{R}^\times$. Second, the \emph{compact real form} $\Fc{n}$ whose fiber at $\varepsilon\ne0$ is $\Mc{n+2}$, the closure of configurations of $n+2$ points on the circle $S^1\subset \mathbb{C}^\times$. Moreover, by \cite[Theorems 9.23, 9.24]{IKLPR}, $\overline F_n(\BR) $ is a deformation retract of both $ \Fc{n} $ and  $ \Fs{n} $.  


\subsection{The cactus group and its variants} For any $1\le i<j\le n$, let $w_{ij}$ be the involution in the symmetric group $S_n$ reversing the interval $[i,j]$. The \emph{cactus group} $ C_n $ is defined to be the group with generators $ s_{ij} $ for $ 1 \le i< j \le n $ and relations
	 \begin{enumerate}
	\item $ s_{ij}^2 = 1 $
	\item $ s_{ij} s_{kl} = s_{kl} s_{ij} $ if $ [i,j] \cap [k,l] = \emptyset $
	\item $ s_{ij} s_{kl} = s_{ w_{ij}(l)  w_{ij}(k)} s_{ij} $ if $ [k,l] \subset [i,j] $
\end{enumerate}
We have an obvious epimorphism $C_n\to S_n$ taking $s_{ij}$ to $w_{ij}$. From \cite{DJS}, we have an isomorphism $ \pi_1^{S_n}(\overline M_{n+1}(\BR)) \cong C_n $ such that the above projection $C_n\to S_n$ is the natural homomorphism $ \pi_1^{S_n}(\overline M_{n+1}(\BR))\to \pi_1^{S_n}(pt)$.

In this article, we consider three variants of the cactus group which are the equivariant fundamental groups of the spaces $ \overline F_n(\BR), \Ms{n+2}, \Mc{n+2}$.

The \emph{virtual cactus group} $vC_n$ is generated by a copy of the cactus group $ C_n$ and the symmetric group $S_n $ subject to the relations
$$ w s_{ij} = s_{w(i) w(j)} w, \text{ if $ w \in S_n $ and $ w(i+k) = w(i) + k $ for $ k = 1, \dots, j -i $}$$
According to \cite{IKLPR}, we have $ \pi_1^{S_n}(\overline F_n(\BR)) \cong vC_n $.  

The \emph{mirabolic cactus group} $MC_n$ is the preimage of $ S_n\subset S_{n+1}$ in $ C_{n+1}$. Moreover, as $ \pi_1^{S_{n+1}}(\Ms{n+2}) \cong C_{n+1}$, we have $ \pi_1^{S_n}(\Ms{n+2}) \cong MC_n$.  In \cite{GHR}, it is proven that $ MC_n$ is generated by certain elements $ s_{ij} $ for $ 1 \le i < j \le n $ and $ t_i $ for $ i = 1, \dots, n-1 $.

The \emph{affine cactus group} $ AC_n$ has generators $ s_{ij} $ for $ 1 \le i \ne j \le n $ (corresponding to intervals in the cyclic order on $ \BZ/n $). The \emph{extended affine cactus group} $ \widetilde{AC}_n $ is the semidirect product with $ \BZ/n $ whose generator $c$ satisfies $cs_{ij}c^{-1}=s_{i+1,j+1}$. Clearly,  $ \widetilde{AC}_n $ is generated by $c$ and $s_{ij}$ for $ 1 \le i < j \le n $. According to \cite{IKLPR}, we have an isomorphism $ \pi^{S_n}_1( \Mc{n+2} ) \cong \widetilde{AC}_n $.

We have the following statement relating together all the above groups; the first statement is from \cite{IKLPR} (see Theorem \ref{th:AC} below) and the second is proven here as Theorem \ref{th:pi1split}.
\begin{thm}\begin{enumerate}
\item \cite{IKLPR} The homomorphism $$ \widetilde{AC}_n \cong \pi^{S_n}_1( \Mc{n+2} ) \rightarrow \pi^{S_n}_1( \Fc{n}) \cong \pi^{S_n}_1(\overline F_n(\BR)) \cong vC_n $$ given by the retraction takes $s_{ij}\in \widetilde{AC}_n$ with $ 1 \le i < j \le n $ to $s_{ij}\in vC_n$ and $c\in\widetilde{AC}_n$ to the long cycle $c\in S_n\subset vC_n$.
\item The homomorphism $$ MC_n \cong \pi^{S_n}_1( \Ms{n+2} ) \rightarrow \pi^{S_n}_1( \Fs{n}) \cong \pi^{S_n}_1(\overline F_n(\BR)) \cong vC_n $$ given by the retraction takes $s_{ij}\in MC_n$ to $s_{ij}\in vC_n$ and $t_i\in MC_n$ to $w_{i,i+1}\in S_n\subset vC_n$.
    \end{enumerate}
\end{thm}

\subsection{Concrete coboundary monoidal categories} 

Recall, from \cite{HK}, that a \emph{coboundary category} is a monoidal category $ \mathcal{C} $ along with a natural isomorphism called \emph{commutor} $$ \sigma_{A, B} : A \otimes B \rightarrow B \otimes A$$
	satisfying the following two axioms.
	\begin{enumerate}
		\item For all $A, B \in \mathcal{C}$, we have $\sigma_{B, A} \circ \sigma_{A,B} = id_{A \otimes B} $
		\item For all $ A, B, C \in \mathcal{C} $ we have $\sigma_{A, B \otimes C}\circ (id \otimes\sigma_{B, C}) = \sigma_{A \otimes B, C}\circ (\sigma_{A, B} \otimes id)$.
	\end{enumerate}

All possible commutors on the $n$-fold tensor product	$X_1\otimes\ldots\otimes X_n$ in a coboundary category $\mathcal{C}$ generate an action of the cactus group $C_n$. 

Similarly, the \emph{virtual cactus group} $vC_n=\pi_1^{S_n}(\overline F_n)$ arises in \emph{concrete} coboundary monoidal categories. Namely, suppose we are given a coboundary monoidal category $\mathcal{C}$ and a faithful monoidal functor $F:\mathcal{C}\to Sets$. Then for any $n$-tuple of objects $X_1,\ldots,X_n\in \mathcal{C}$, the set $F(X_1\otimes\ldots\otimes X_n)=F(X_1)\times\ldots\times F(X_n)$ is naturally acted on by $vC_n$, where the symmetric group $S_n=\pi_1^{S_n}(pt)$ acts by naive permutations of factors (i.e. by commutors in the category of sets) and $C_n$ acts by commutors in the category $\mathcal{C}$. 

So, given a concrete coboundary monoidal category $\mathcal{C}$, any set $F(X_1\otimes\ldots\otimes X_n)=F(X_1)\times\ldots\times F(X_n)$ can be regarded as a covering of the cactus flower space $\overline{F}_n(\BR)$, by assigning $F(X_1\otimes\ldots\otimes X_n)$ as the fiber over the flower point $\infty\subset\overline{F}_n$. Suppose that $\mathcal{C}$ is \emph{$\Xi$-coloured} (that is semisimple with the isomorphism classes of simple objects indexed by the set $\Xi$). Then the above coverings for all possible $n$ and all possible collections of objects $X_1,\ldots,X_n$ are compatible with the stratification of $\overline{F}_n(\BR)$: we axiomatize such \emph{$\Xi$-coloured operadic coverings} in Section~\ref{se:compcover}. In the present paper, we prove the following statement, generalizing \cite[Theorem~4.16]{HKRW}.

\begin{thm}\label{th:operadic-intro}
    We have an equivalence of categories between the category of $\Xi$-coloured operadic coverings of the moduli spaces of cactus flower curves and the category of $\Xi$-coloured concrete coboundary categories.  
\end{thm}

The main example of a coboundary monoidal category is the category of Kashiwara $\fg$-crystals (a combinatorial version of the category of finite-dimensional $\fg$-modules) for any complex simple Lie algebra $ \fg$. More precisely, to any irreducible representation $V(\lambda)$ one assigns an oriented graph $B(\lambda)$ (called the normal crystal of highest weight $\lambda$) whose vertices correspond to basis vectors of $V(\lambda)$ and are labeled by the weights of $V(\lambda)$, while the edges correspond to the action of the Chevalley generators and are labeled by the simple roots of $\fg$.  There is a purely combinatorial rule of tensor multiplication that provides a structure of a crystal on any Cartesian product $B(\lambda_1)\times B(\lambda_2)$.  This is a concrete coboundary monoidal category with $F$ being the forgetful functor. According to \cite{HKRW}, $\fg$-crystals as a coboundary monoidal category is determined by the operadic covering of $\overline M_{n+1}(\BR)$ formed by \emph{Bethe eigenlines for the Gaudin model} attached to $\fg$. In the present paper we extend this statement to $\fg$-crystals as a \emph{concrete} coboundary monoidal category and the operadic covering of moduli spaces of cactus flower curves formed by Bethe eigenlines for the \emph{inhomogeneous} Gaudin model.

\subsection{Gaudin models}
Let $U\fg$ be the universal enveloping algebra of $ \fg$. We consider some remarkable families of large commutative subalgebras in $U\fg^{\otimes n}$ whose construction goes back to Gaudin \cite{g1,g2} (for $\fg=\mathfrak{sl}_N$) so we call them \emph{Gaudin algebras}. In \cite{ffre}, Feigin, Frenkel and Reshetikhin gave the most general construction of such algebras from the center of the completed universal enveloping algebra of the affine Lie algebra $\hat{\fg}$ at the critical level. 

In \cite{IKR}, we considered three versions of the Gaudin models: rational homogeneous, rational inhomogeneous, and trigonometric, all obtained by certain versions of the Feigin-Frenkel-Reshetikhin construction.

\subsubsection{Rational homogeneous Gaudin model} The Gaudin  algebra $ \A(\uz) $ is a maximal commutative subalgebra of $ (U(\fg)^{\otimes n})^\fg $, which depends on an $ n $-tuple of distinct complex numbers $\uz=(z_1, \dots, z_n)$.  The quadratic part of the Gaudin algebras is spanned by the Gaudin Hamiltonians which are simple expressions involving the Casimir; the full definition of the algebra is more complicated and involves the Feigin-Frenkel centre of the completed universal enveloping of the affine Lie algebra $\hat{\fg}$ at the critical level. 

Fix $ \lambda_1, \dots, \lambda_n $ dominant weights and consider the tensor product of irreducible representations $ V(\ul) := V(\lambda_1) \otimes \cdots \otimes V(\lambda_n) $. The homogeneous Gaudin algebra commutes with the diagonal copy of $ \fg $, so it acts on any hom-space $ V(\ul)^\mu := \Hom_\fg(V(\mu),V(\lambda_1)\otimes\ldots\otimes V(\lambda_n)) $ for any collection of highest weights $\ul=(\lambda_1,\ldots,\lambda_n)$ and $\mu$.  

The tuple $ \uz = (z_1, \dots, z_n) $ can be regarded as a point in $ \overline{M}_{n+1} $ by taking $ \PP^1 $ and marking the points $ z_1, \dots, z_n $ and $ \infty $; in this way such tuples correspond to the non-boundary points of $ \overline{M}_{n+1}$.  In \cite{AFV}, Aguirre-Felder-Veselov proved that the family of subspaces spanned by the quadratic Gaudin Hamiltonians extends to a family parametrized by $ \overline{M}_{n+1} $. In \cite{Rcactus}, the same was proved for the whole subalgebras $\A(\uz)$.

\subsubsection{Rational inhomogeneous Gaudin model} The inhomogeneous Gaudin subalgebras $ \mathcal{A}_\chi(\uz) $, for regular $\chi\in\fh$, are maximal commutative subalgebras in $U\fg^{\otimes n}$. The subalgebras $ \mathcal{A}_\chi(\uz) $ are known to commute with the diagonal copy of the Cartan subalgebra $\fh\subset\fg$, so they act on $n$-fold tensor products of $\fg$-modules preserving weight spaces. In particular, the action of $ \mathcal{A}_\chi(\uz) $ on the tensor product of irreducible representations $ V(\ul) = V(\lambda_1) \otimes \cdots \otimes V(\lambda_n) $ preserves the weight spaces $V(\ul)_\mu$ and generically acts without multiplicities on any such space. Fix $\chi$ enough generic and consider $ \mathcal{A}_\chi(\uz) $ as a family of subalgebras depending on the parameter $\uz$ up to simultaneous additive shift of all the $z_i$'s. In \cite{IKR}, we extended this family to $\overline{F}_n$. 

\subsubsection{Trigonometric Gaudin model} In \cite{IKR}, we defined the trigonometric Gaudin subalgebras $\mathcal{A}^{trig}_\theta(\uz)\subset U\fg^{\otimes n}$. In type A, they were previously considered in \cite{mtv} and \cite{mr}. Fix $\uz=( z_1, \dots, z_n) \in \mathbb{C}^\times$ distinct and $ \theta \in \fh $. Similarly to inhomogeneous Gaudin subalgebras, $ \mathcal{A}^{trig}_\theta(\uz) $ commutes with the diagonal copy of the Cartan subalgebra $\fh\subset\fg$, so it acts on any weight space $ V(\ul)_\mu $. Similarly to the rational inhomogeneous case, if we fix $\theta$ enough generic, we get the family $ \mathcal{A}^{trig}_\theta(\uz) $ of subalgebras depending on the parameter $\uz$, but now up to simultaneous multiplicative shift of all the $z_i$'s. In \cite{IKR}, we extended this family to $\overline{M}_{n+2}$ and explained that in fact it is obtained from the family of homogeneous rational Gaudin algebras by certain quantum Hamiltonian reduction.

\subsubsection{Degeneration of trigonometric Gaudin subalgebras to rational inhomogeneous} 
In \cite{IKR}, we proved that the trigonometric Gaudin algebras can degenerate to the inhomogeneous ones.  Namely, trigonometric and rational inhomogeneous (with $ \theta$ fixed) combine together into a single family of commutative subalgebras in $U\fg^{\otimes n}$ parametrized by the space $ \overline \CF_n$ introduced in \cite{IKLPR}.  
This is the total space of a degeneration of $ \overline M_{n+2}$ to $ \overline F_n $.   

\subsection{Monodromy of Bethe eigenlines} In \cite{HKRW}, along with Halacheva and Weekes, we studied the covering of $\overline{M}_{n+1}(\BR)$ by joint eigenlines for the action of homogeneous Gaudin algebras and related it to the action of the cactus group $ C_n = \pi_1^{S_n}(\overline M_{n+1}(\BR))$ on tensor products of crystals. Namely, we proved that for all values of the parameter $\uz\in\overline{M}_{n+1}$, the joint eigenvalues of $\mathcal{A}(\uz)$ on $\Hom_\fg(V(\mu),V(\ul))$ are all different, and, furthermore, there is a bijection between those joint eigenvalues and the multiplicity set of $ B(\mu)$ in the tensor product of $\fg$-crystals $B(\lambda_1)\otimes\ldots\otimes B(\lambda_n)$, compatible with the $C_n$-action. 

Similarly, in \cite{IKR}, we proved that for fixed real $\chi$ (respectively, under appropriate assumptions on $\theta$) the inhomogeneous (respectively, trigonometric) Gaudin algebra acts without multiplicities on the weight space $V(\ul)_\mu$ for all values of the parameter from $\overline{F}_n(\BR)$ (respectively, from $\Ms{n+2}$ and $\Mc{n+2}$). The main result of the present paper is the following statement conjectured in \cite{IKR}.

\begin{thm}\label{th:main-intro}
\begin{enumerate} \item The monodromy action of $ vC_n = \pi_1^{S_n}(\overline F_n(\mathbb{R}))$ on the inhomogeneous Gaudin eigenlines in $ V(\ul)_\mu$ matches the action of $ vC_n$ on $ \mathcal B(\lambda_1) \times \cdots \times \mathcal B(\lambda_n)$ given by crystal commutors and naive permutation of tensor factors.

\item The monodromy action of $MC_n=\pi_1^{S_n}(\Ms{n+2})$ on the trigonometric Gaudin eigenlines in $V(\ul)_\mu$ factors surjectively through the action of $ vC_n$.
\item
The monodromy action of $ \widetilde{AC}_n = \pi_1^{S_n}(\Mc{n+2})$ on the trigonometric Gaudin eigenlines in $ V(\ul)_\mu$ factors through the action of $ vC_n $ and is given by crystal commutors and cyclic rotation of tensor factors.
\end{enumerate}
\end{thm}

\subsection{Combinatorial wall-crossing for minuscule resolutions of slices in the affine Grassmannian} 
 According to the general philosophy of Bezrukavnikov and Okounkov, we expect that the above monodromy theorem can be put into the following general context as the combinatorial counterpart of \emph{wall-crossing functors} between the derived categories $\mathcal{O}$ of modular representations of the quantization of a conical symplectic singularity (see \cite[Sections~1.8-1.11]{IKR} for more detailed discussion). 

Let $Y$ be a symplectic resolution of a conical symplectic singularity with the action of the torus $T$. Then the small quantum $T$-equivariant cohomology ring is a family of commutative algebras acting on the equivariant cohomology $H^\bullet_{T\times \Cx}(Y)$. The space of parameters is the complement of some divisor in the torus $T'=\Hom(H_2(Y,\BZ),\Cx)$ formed by cosets by certain algebraic subtori in $T'$. The closure of this family of algebras is some compactification of this complement (in all known examples it is an appropriate version of the de Concini-Procesi wonderful closure). We expect that for real values of the parameters (in the compact real form) the corresponding algebra acts semisimply with simple spectrum. Thus we get an action of the fundamental group of the closure on the set of eigenvectors. The divisor splits the compact real form of $T'$ into a disjoint union of \emph{alcoves}, so the corresponding real locus of the de Concini-Procesi compactification is the union of the same alcoves with a normal crossings boundary divisor. Thus we can view this fundamental group as a groupoid whose objects are the alcoves and morphisms are paths between the interiors of the alcoves transversal to the boundary. Following Bezrukavnikov and Okounkov, we conjecture that this action comes from the action of the fundamental groupoid of the space of real stabilities of Anno, Bezrukavnikov and Mirkovic \cite{ABM} on the $K$-group of the category of coherent sheaves on $Y$ coming from \emph{wall-crossing functors}, via an appropriate analog of Drinfeld's unitarization procedure and taking the $\hbar\to0$ limit. The setting of \cite{ABM} looks similar because the real space of stabilities is the union of the same alcoves, and the transitions across the codimension~$1$ strata of the De Concini--Procesi closure are perverse equivalences, as it is explained by Losev in \cite[section 10.1]{Losev}. More precisely, let $A$ be an alcove in the space of real stabilities and let $\mathcal{E}(A)$ be the set the joint eigenlines for the operators of quantum multiplication by equivariant cohomology of $Y$ on $H^\bullet_T(Y)$, with the quantum parameter from the alcove $A$. 
\begin{conj}
There is a natural indexing of the set $\mathcal{E}(A)$ by simple objects of the abelian category (i.e. heart of the $t$-structure) corresponding to the alcove $A$ such that the monodromy of $\mathcal{E}(A)$ along the paths crossing the boundary divisor transversely are given by the permutations of simple objects determined by the corresponding perverse equivalence.
\end{conj}

Let $\lambda$ be a dominant $\fg$-weight expressible as a sum of minuscule weights $\lambda=\lambda_1+\ldots+\lambda_n$.  Then for any $\fg$-weight $\mu$ that occurs in the representation $V(\lambda)$, there is a transversal slice in the affine Grassmannian of the Langlands dual group ${}^LG$, namely, the slice $\mathrm{Gr}^\lambda_\mu$ to the ${}^LG[[t]]$-orbit of $t^\mu$ in the closure of the ${}^LG[[t]]$-orbit of $t^\lambda$. This is a conical symplectic singularity admitting a symplectic resolution $ Y = \mathrm{Gr}^\ul_\mu $ determined by the collection of minuscule weights $\ul:=(\l_1,\ldots,\l_n)$. According to Ginzburg and Riche \cite{gr}, the (specialized) equivariant cohomology space of the minuscule resolution $ Y = \mathrm{Gr}^\ul_\mu $ of a slice $\mathrm{Gr}^\lambda_\mu$ in the affine Grassmannian related to a simply-laced Lie algebra ${}^L\fg$ is the $\mu$-th weight space $V(\ul)_\mu$ in the tensor product of minuscule representations $V(\lambda_1)\otimes\ldots\otimes V(\lambda_n)$. Next, according to Danilenko \cite{Dan} the quantum connection in the equivariant cohomology $ Y $ is the trigonometric KZ connection attached to $\fg$.  In fact, Danilenko proved that the operators of quantum multiplication by divisors are given by the quadratic trigonometric Gaudin Hamiltonians, where $\theta\in\fh^*$ is the equivariant parameter interpreted as a weight of $\fg$ and $z_1,\ldots, z_n$ are quantum parameters. Since generically the operators of quantum multiplication act with a simple spectrum, this means that, generically, the quantum cohomology ring is the image of the trigonometric Gaudin subalgebra $\mathcal{A}_{\theta}^{trig}(z_1,\ldots,z_n)\subset U(\fg)^{\otimes n}$ in $V(\ul)_\mu$. We expect that the bijections coming from perverse equivalences generate the action of the extended affine cactus group $ \widetilde{AC}_n $ on the corresponding tensor product crystal (see \cite{HLLY}). With this in mind, the third assertion of Theorem~\ref{th:main-intro} is a particular case of the above conjecture, for minuscule resolutions of slices in the affine Grassmannian of a simply-laced ${}^L\fg$. 

\subsection{Organization} The paper is organized as follows. In Section 2, we give a general background on coboundary monoidal categories and show how virtual cactus group arises in \emph{concrete} coboundary categories.  Section 3 contains some generalities on $\fg$-crystals. In section 4 we recall the results about moduli spaces of cactus flower curves. In section 5 we define operadic coverings of these moduli spaces and prove Theorem~\ref{th:operadic-intro} explaining their relation to concrete coboundary categories. In section 6 we recall the main results of \cite{IKR} about rational inhomogeneous and trigonometric Gaudin subalgebras. In section 7 we prove Theorem~\ref{th:main-intro}.

\subsection{Acknowledgments} We thank Alexander Braverman, Iva Halacheva, Aleksei Ilin, and Vasily Krylov for comments and suggestions.  The work of L.R. was supported by the Fondation Courtois. 

\section{Concrete coboundary categories}\label{se:ConcreteCat}
	\subsection{Concrete $\Xi$-coloured categories} 
We will define the notion of a $\Xi$-coloured category, where $ \Xi $ is a set.

\begin{defn}
	Let $ \Xi$ denote an arbitrary set.  We define the category of $\Xi$-graded sets, $ \Xiset $, as follows.  The objects of $ \Xiset$ are sequences $ (A_\l)_{\l \in \Xi} $ where each $ A_\l $ is a finite set, with only finitely many $ A_\l $ non-empty.  The morphisms of $ \Xiset $ are given by
	$$
	\Hom_{\Xiset}(A, B) = \prod_\l \Hom_{\set}(A_\l, B_\l)
	$$
	
	A category $ \cC $ is called \textbf{$\Xi$-coloured}, if we are given an equivalence between $ \cC $ and $ \Xiset $.	
\end{defn}
We write $ L(\l) $ for the object of $ \Xiset $ with $ L(\l)_\l = \{* \} $ and $ L(\l)_\mu = \emptyset $ for $ \mu \ne \lambda $.

 The data of a monoidal structure on $ \Xiset$ is given by sets $  (L(\l_1) \otimes L(\l_2))_\mu $ for each $ \l_1, \l_2, \mu \in \Xi $ so that
$$
(A \otimes B)_\mu = \bigsqcup_{\l_1, \l_2 \in \Xi} (L(\l_1) \otimes L(\l_2))_\mu \times A_{\l_1} \times B_{\l_2}
$$
along with an associator $ \alpha $, which is given by bijections 
$$ \alpha_{\l, \mu, \nu} : \sqcup_{\rho, \gamma} (L(\l) \otimes L(\mu))_\gamma \times (L(\gamma) \otimes L(\nu))_\rho \rightarrow \sqcup_{\rho, \tau} (L(\l) \otimes L(\tau))_\rho \times (L(\mu) \otimes L(\nu))_\tau
	$$

Let $\cC, \cD $ be two monoidal categories.  A monoidal functor is a functor $ \Phi : \cC \rightarrow \cD $ along with a natural isomorphism $ \phi_{A,B}: \Phi(A \otimes B) \rightarrow \Phi(A) \otimes \Phi(B) $.  This natural isomorphism must satisfy a hexagon axiom that relates the associators in $ \cC $ and $ \cD$.  For all $ A, B, C \in \cC$,
	\begin{equation} \label{eq:hex1}
		\begin{tikzcd}
		  & \Phi((A \otimes B) \otimes C) \ar[r,"{\Phi(\alpha_{A, B, C})}", outer sep = 2 pt] \ar[dl,"{\phi_{A\otimes B, C}}", outer sep = -2 pt] & \Phi(A \otimes (B \otimes C)) \ar[dr,"{\phi_{A, B \otimes C}}", outer sep = -2 pt]& \\
				\Phi(A \otimes B) \otimes C \ar[dr,"{\phi_{A,B} \otimes id}", outer sep = -2 pt] & & & \Phi(A) \otimes \Phi(B \otimes C)  \ar[dl,"{id \otimes \phi_{B, C} }", outer sep = -2 pt] \\
				& (\Phi(A) \otimes \Phi(B)) \otimes \Phi(C) \ar[r,"{\alpha_{\Phi(A), \Phi(B), \Phi(C)}}", outer sep = 2 pt] & \Phi(A) \otimes (\Phi(B) \otimes \Phi(C) )&
		\end{tikzcd}
		\end{equation}

A $\Xi$-coloured monoidal category is called \textbf{concrete} if we are given a monoidal, faithful, functor $ \Phi : \cC \rightarrow \set $.  Using the equivalent between $ \cC $ and $ \Xiset $, it is easy to see that such a monoidal functor is given by the following data
\begin{enumerate}
	\item For each $ \l \in \Xi $, a non-empty set $ \CL(\l) := \Phi(L(\l))$.  The value of $ \Phi $ on $ (A_\l) $ is given by $ \sqcup A_\l \times V(\l)$.
	\item For each pair $ \l_1, \l_2 \in \Xi $, a bijection $  \phi_{\l_1, \l_2} : \sqcup_\mu (L(\l_1) \otimes L(\l_2))_\mu \times \CL(\mu)  \rightarrow \CL(\l_1) \times \CL(\l_2)  $. This is the isomorphism $ \Phi(L(\l_1) \otimes L(\l_2)) \rightarrow \Phi(L(\l_1)) \times \Phi(L(\l_2)) $, which is the monoidal structure on the functor $ \Phi$.
\end{enumerate}

\subsection{Concrete coboundary categories}
	Recall that a \textbf{coboundary category} is a monoidal category $ \cC $ along with a natural isomorphism $$ \sigma_{A, B} : A \otimes B \rightarrow B \otimes A$$
	satisfying the following two axioms.
	\begin{enumerate}
		\item For all $A, B \in \cC$, we have $\sigma_{B, A} \circ \sigma_{A,B} = id_{A \otimes B} $
		\item For all $ A, B, C \in \cC $, the following hexagon commutes
		\begin{equation} \label{eq:hex}
		\begin{tikzcd}
		  & (A \otimes B) \otimes C \ar[r,"{\alpha_{A, B, C}}"'', outer sep = 2 pt] \ar[dl,"{\sigma_{A\otimes B, C}}"', outer sep = -2 pt] & A \otimes (B \otimes C) \ar[dr,"{\sigma_{A, B \otimes C}}", outer sep = -2 pt]& \\
				C \otimes (A \otimes B) \ar[dr,"{id \otimes \sigma_{A,B}}"', outer sep = -2 pt] & & & (B \otimes C) \otimes A \ar[dl,"{\sigma_{B, C} \otimes id}", outer sep = -2 pt] \\
				& C \otimes (B \otimes A) \ar[r,"{\alpha_{C, B, A}^{-1}}"', outer sep = 2 pt] & (C \otimes B) \otimes A &
		\end{tikzcd}
		\end{equation}
	\end{enumerate}

If we have a $ \Xi$-coloured concrete coboundary category, then for any $ \l, \mu \in \Xi $, we can consider the bijection $$\Phi(\sigma_{\l, \mu}): \CL(\l) \times \CL(\mu) \rightarrow \CL(\mu) \times \CL(\l) $$  Note that in general this will be different from the bijection $ \CL(\l) \times \CL(\mu) \rightarrow \CL(\mu) \times \CL(\l) $ given by $ (a,b) \mapsto (b,a)$.  So the hexagon equation (\ref{eq:hex1}) forces a relation between the associators in $ \cC $ and $\Xiset$, but the commutors in these two categories are unrelated.  

To formalize this section, we can consider the category $ \XiCCob$ of all $\Xi$-coloured concrete coboundary categories.  By the above analysis, an object $\cC $ of $ \XiCCob$ is the data of
\begin{equation} \label{eq:XiCCob} \begin{gathered}
\text{sets } (L(\la_1) \otimes L(\la_2))_\mu \text{ for $ \la_1, \la_2, \mu \in \Xi$, and } \CL(\l) \text{ for $ \la \in \Xi$, along with bijections}\\
 \alpha_{\l, \mu, \nu} : \sqcup_{\rho, \gamma} (L(\l) \otimes L(\mu))_\gamma \times (L(\gamma) \otimes L(\nu))_\rho \rightarrow \sqcup_{\rho, \tau} (L(\l) \otimes L(\tau))_\rho \times (L(\mu) \otimes L(\nu))_\tau \\ \phi_{\l_1, \l_2} : \sqcup_\mu (L(\l_1) \otimes L(\l_2))_\mu \times \CL(\mu)  \rightarrow \CL(\l_1) \times \CL(\l_2) \\ \Phi(\sigma_{\l, \mu}): \CL(\l) \times \CL(\mu) \rightarrow \CL(\mu) \times \CL(\l)
\end{gathered}
\end{equation}
which must satisfy the following conditions:
\begin{enumerate}
    \item The pentagon axiom for the associator $ \alpha$.
    \item The hexagon axiom (\ref{eq:hex1}) relating the associator $\alpha$ and the monoidal structure $ \phi$.
    \item The involutive axiom for the commutor $\Phi(\sigma_{\l, \mu})$ and the hexagon axiom (\ref{eq:hex}) relating the assoicators and the commutors.
\end{enumerate}

A morphism in the category $ \XiCCob$ is a functor $ \cC \rightarrow \cD $ compatible with all this structure, including the equivalences with $ \Xiset$.  Thus, such a morphism is given by a map between the above collections of sets, compatible with above bijections (\ref{eq:XiCCob}).

\subsection{The virtual cactus group} \label{se:vCn}
 The \textbf{cactus group} $ C_n $ is defined to be the group with generators $ s_{ij} $ for $ 1 \le i< j \le n $ and relations
	 \begin{enumerate}
	\item $ s_{ij}^2 = 1 $
	\item $ s_{ij} s_{kl} = s_{kl} s_{ij} $ if $ [i,j] \cap [k,l] = \emptyset $
	\item $ s_{ij} s_{kl} = s_{ w_{ij}(l)  w_{ij}(k)} s_{ij} $ if $ [k,l] \subset [i,j] $
\end{enumerate}
Here $ w_{ij} \in S_n $ is the element of $ S_n $ which reverses $ [i,j] $ and leaves invariant the elements outside this interval.  We write $ w_i = w_{i i+1} $, the usual elementary transposition.

From \cite[\S 20.5]{IKLPR}, we recall the definition of the virtual cactus group. Suppose that $1 \le  i< j \le n$.  We say that $ w \in S_n $ is a \textbf{translation} on $[i,j] $, if $ w(i+k) = w(i)+k$ for $ k = 1,\dots, j-i$.  

\begin{lem} \label{le:cable} Fix $ 1 \le  i< j \le n$.
    There is a bijection (called ``cabling'') 
    $$ S_{n - (j-i)} \longrightarrow \bigl\{ w \in S_n : w \text{ is a translation on } [i,j]  \bigr\}$$  taking $ u \in S_{n-(j-i)} $ to $ w \in S_n $, where $ w $ is defined by
    $$
 w(a) = \begin{cases}  u(a) \text{ if $ a < i $ and $ u(a) < p $ } \\
 u(a) + q  \text{ if $ a < i $ and $ u(a) > p $ } \\
 p + a - i \text{ if  $ i \le a \le j $  } \\
 u(a - q ) \text{ if $ a > j $ and $ u(a) < p $ } \\
 u(a-q) + q  \text{ if $ a > j $ and $ u(a) > p $ } 
 \end{cases}
$$
where $ q = j-i $ and $ p = w(i)$.
\end{lem}

We define the \textbf{virtual cactus group} $vC_n $ to be the quotient of the free product $ C_n * S_n $ by the relations
$$
w s_{ij} w^{-1} = s_{w(i) w(j)}
$$
for all $ 1 \le i < j \le n $ and $ w \in S_n $, such that $ w $ is a translation on $[i,j]$.

There is a homomorphism $ vC_n \rightarrow S_n $ which takes $ s_{ij} $ to $ w_{ij} $ and is the identity on $ S_n $.

\subsection{Virtual cactus group action}

In \cite{HK}, we proved that $C_n $ acts on tensor products in coboundary categories.  We now extend this result and show that $ vC_n $ acts on the images under $ \Phi$ of tensor products in concrete coboundary categories.

\begin{thm} \label{th:vCnact}
Let $ \cC $ be a $\Xi$-coloured concrete coboundary category.  Then there is an action of $ vC_n $ on $ \sqcup_{\ul \in \Xi^n} \CL(\l_1) \times \cdots \times \CL(\l_n) $, where $ vC_n $ acts on $ \Xi^n $ via the homomorphism $ vC_n \rightarrow S_n $.
\end{thm}

\begin{proof}
From \cite{HK}, we have an action of $ C_n$ on $ \sqcup_{\ul} L(\l_1) \otimes \cdots \otimes L(\l_n) $.  Applying $ \Phi $ we get an action of $ C_n $ on $ \sqcup_{\ul \in \Xi^n} \CL(\l_1) \times \cdots \times \CL(\l_n) $.

On the other hand, there is an action of $ S_n $ on $  \sqcup_{\ul \in \Xi^n} \CL(\l_1) \times \cdots \times \CL(\l_n) $ by permutations.  We claim that together these actions generate an action of $ vC_n$.  To check this, let $ 1\le i , j \le n $ and let $ w $ be a translation on $ [i,j] $.  Consider $ \CL(\l_1) \times \cdots \times \CL(\l_n) $.  Let $ X = L(\l_i) \otimes \cdots \otimes L(\l_j)$.  We can write 
$$ \CL(\l_1) \times \cdots \times \CL(\l_n) = \CL(\l_1) \times \cdots \times \Phi(X) \times \cdots \CL(\l_n)  $$
Since $ w $ is a translation on $ [i,j] $, there exists $ u \in S_{n-(j-i)} $ (as above) such that the action of $ w $ on $\CL(\l_1) \times \cdots \times \CL(\l_n)$ is given by the action of $ u $ on $
 \CL(\l_1) \times \cdots \times X \times \cdots \CL(\l_n)$.  Now the action of $ s_{ij} $ on $ L(\l_1) \otimes \cdots \otimes L(\l_n) $ is given by $ 1 \otimes \sigma \otimes 1 $ where $ \sigma :  X \rightarrow X  $ is the long element of the cactus group acting on $ X $.  Hence by the naturality of $ u $, we see that $ u s_{ij} u^{-1} = 1 \times \Phi(\sigma) \times 1$.  But now since $ X $ lies in positions $ w(i), \dots, w(j) $, we see that $ w s_{ij} w^{-1} =  s_{w(i) w(j)} $ as desired.
 
\end{proof}

\section{Crystals}
\subsection{The category of normal crystals}
The main example of a concrete coboundary category is the category of crystals. 	Recall that we fixed a semisimple Lie algebra $ \fg $ with Cartan subalgebra $ \fh $.  Let $ \Lambda \subset \fh^* $ denote the weight lattice and let $ \Lambda_+ $ denote the set of dominant weights.  We fix root vectors $ e_\alpha, f_\alpha  $ for each $ \alpha \in \Delta_+ $.  For each $ \lambda \in \Lambda_+ $, we write $ V(\l) $ for the irreducible representation of $ \fg $ of highest weight $ \l $.
	
	A crystal $ B $ is a finite set which models a weight basis for a representation of $ \fg $.  A crystal has two structures: a weight map, which tells us in which weight space our basis vector lives, and partially defined maps $ e_i, f_i $ for $ i \in I $ which indicate the leading order behaviour of the simple root vectors $ e_{\alpha_i}, f_{\alpha_i} $ on the basis.  Here is the precise definition.
	
	\begin{defn}
		A $\fg$-\emph{crystal} is a finite set $ B $ along with a map $ wt : B \rightarrow \Lambda $ and maps $ e_i, f_i : B \rightarrow B \sqcup \{0\} $ for each $ i \in I $ satisfying the following conditions
		\begin{enumerate}
			\item For each $ i \in I, b \in B $, if $ e_i(b) \ne 0 $, then $ wt(e_i(b)) = wt(b) + \alpha_i $
			\item For each $ i \in I, b \in B $, if $ f_i(b) \ne 0 $, then $ wt(f_i(b)) = wt(b) - \alpha_i $
			\item For each $ i \in I, b \in B $, if $ e_i(b) \ne 0 $, then $ f_i(e_i(b)) = b $
			\item For each $ i \in I, b \in B $, if $ f_i(b) \ne 0 $, then $ e_i(f_i(b)) = b $
		\end{enumerate}
	\end{defn}
	
	
	If $ B_1, B_2 $ are two crystals, then their disjoint union $ B_1 \sqcup B_2 $ carries a crystal structure, where $ e_i, f_i $ are defined in the natural way.
	
	
	
	A morphism between two crystals $ B_1, B_2 $ is a map $ \psi : B_1 \rightarrow B_2 $ commuting with the structure maps $ wt, e_i, f_i $.  (This is sometimes called a strict morphism elsewhere in the literature.)
	
	
	
	There is a construction which assigns a crystal to each representation of $ \fg $ (in fact, there are several such constructions, but they all have the same output).  For each dominant weight $ \lambda $, we let $ B(\lambda) $ denote the crystal of the representation $ V(\l) $.
	
	\begin{defn}
	A crystal $ B $ is called \emph{normal} if it is isomorphic to a disjoint union of the crystals $B(\l) $ (equivalently if it is the crystal of a representation of $\fg$). We let $ \gcrys$ denote the category of normal $\fg $-crystals.
	\end{defn}
	
			
	
The following result is standard (see \cite[Cor 5.5]{HKRW}).
	\begin{thm}
	The category $ \gcrys$ is equivalent to the category  $ \Lset$ via the functor
	\begin{align*}
		\Lset &\rightarrow \gcrys\\
		(A_\lambda)_{\lambda \in \Lambda_+} &\mapsto \sqcu_\lambda A_\lambda \times B(\lambda)
	\end{align*}
	\end{thm}
	
	\subsection{Tensor products}
	\begin{defn} \label{Def:Tensor}
	If $ B_1, B_2 $ are two crystals, then we define $ B_1 \otimes B_2 $ to be the crystal whose underlying set is $ B_1 \times B_2 $ and whose structure maps are defined by
	\begin{align*}
		wt(b_1, b_2) &= wt(b_1) + wt(b_2)\\
		e_i(b_1, b_2) &= \begin{cases} (e_i(b_1),b_2), & \; \text{if} \; \varepsilon_i(b_1) > \varphi_i(b_2) \\
			(b_1,e_i(b_2)), & \; \text{otherwise}
		\end{cases} \\
		f_i(b_1, b_2) &= \begin{cases} (f_i(b_1),b_2), & \; \text{if} \; \varepsilon(b_1) \geq \varphi(b_2) \\
			(b_1,f_i(b_2)), & \; \text{otherwise}
		\end{cases}
	\end{align*}
	\end{defn}

	The following results are well-known, see \cite{HK}.
	\begin{prop}
	\begin{enumerate}
	\item The tensor product of normal crystals is again a normal crystal and moreover the decomposition of $ B(\l_1) \otimes B(\l_2) $ matches that of $V(\l_1) \otimes V(\l_2) $.
	\item The tensor product of crystals is ``associative'' in the sense that the obvious map
	\begin{align*}
		\alpha : (B_1 \otimes B_2) \otimes B_3 &\rightarrow  B_1 \otimes (B_2 \otimes B_3) \\
		((b_1, b_2), b_3) &\mapsto (b_1, (b_2, b_3))
	\end{align*}
	is an isomorphism of crystals.  This associator obviously satisfies the pentagon axiom.
    \item The category of normal crystals carries a commutor defined using the Schutzenberger involution.  
	\end{enumerate}
	Thus $ \gcrys$ is a coboundary category.
	\end{prop}
	
	Because the associator for $ \gcrys $ is so simple, we will write multiple tensor products of crystals without brackets.

On the other hand, by construction $ \gcrys$ is a concrete category and thus $ \gcrys $ is a $ \Xi$-coloured concrete coboundary category.  From Theorem \ref{th:vCnact}, we immediately deduce the following result, where $ \mathcal B(\l)$ denotes the underlying set of the crystal $ B(\l)$.

\begin{cor}
    There is an action of $ vC_n $ on $ \sqcup_{\ul} \mathcal B(\l_1) \times \cdots \times \mathcal B(\l_n)$.  The subgroup $S_n \subset vC_n$ acts by ordinary permutations of factors, while $ C_n \subset vC_n $ acts by crystal commutors.
\end{cor}

\subsection{An example}
The action of $ vC_n $ on $ \sqcup_{\ul} \mathcal B(\l_1) \times \cdots \times \mathcal B(\l_n) $ is highly non-trivial combinatorially.  We will illustrate this with an example.  Let us take $ \fg = \mathfrak{sl}_n $ and consider the action of $ vC_n$ on $ \mathcal B(\omega_1)^n$.  The action of $ vC_n $ preserves the weight function and so $ vC_n $ acts on the weight subset $ \mathcal B(\omega_1)^n_{(1,\dots, 1)}$.  We identify $ \mathcal B(\omega_1) = \{1, \dots, n\} $.  In this way, elements of $ \mathcal B(\omega_1)^n_{(1,\dots, 1)}$ are identified with sequences $ (a_1, \dots, a_n) $ which form a permutation of $ \{1, \dots, n \} $.  In other words, we have a bijection $ \mathcal B(\omega_1)^n_{(1,\dots, 1)} = S_n$.

Recall the celebrated RSK correspondence
$$
S_n \xrightarrow{\sim} \sqcup_{\l} Tab(\l) \times Tab(\l)
$$
where the union is taken over partitions of $ n $ and $ Tab(\l) $ denotes the set of standard Young tableaux of shape $ \l$.  We have a well-studied action of the cactus group $ C_n$ on $ Tab(\l) $, via partial Sch\"utzenberger involutions, originally due to Berenstein-Kirillov \cite{BK} (see also \cite{CGP}). In \cite{LR}, Sophia Liao and the second author proved that for all non-hook-shaped and non-self-transpose $\l$, the image of $ C_n$ in the permutation group of $ Tab(\l) $ determined by this action contains all even permutations of $ Tab(\l) $.

\begin{prop}
    The action of $ vC_n $ on $\mathcal B(\omega_1)^n_{(1,\dots, 1)} = S_n$ is characterized as follows.
    \begin{enumerate}
        \item The subgroup $ S_n \subset vC_n $  acts on $ S_n $ by left multiplication.
        \item The subgroup $ C_n \subset vC_n $ acts on $ S_n = \sqcup_{\l} Tab(\l) \times Tab(\l)$ by its action on the left factors of $ Tab(\l)$.
    \end{enumerate}
\end{prop}

\begin{proof}
    The first statement is clear by construction.  The second statement follows from Theorem 5.13 of \cite{Hal}.
\end{proof}

	\section{Moduli space of cactus flower curves}

\subsection{Moduli spaces}
Let $ M_{n+1} = (\BP^1)^{n+1} \setminus \Delta / PGL_2 $ be the configuration space of $ n+1 $ points on $ \BP^1$ modulo automorphisms.  By fixing the last marked point (labelled by $ n+1$) to be $ \infty$, we have $ M_{n+1} = \C^n \setminus \Delta / \Cx \ltimes \C$, the configuration space of $ n $ points on the affine line, modulo scaling and translation.  There is a natural Deligne-Mumford compactification $ \overline M_{n+1} $ of $ M_{n+1}$.  The points of $ \overline M_{n+1} $ are stable, nodal, genus 0 curves with $ n+1 $ marked points; we call these \textbf{cactus curves}.

For each triple $ i, j, k $ of distinct elements of $ \{1, \dots, n\} $ and any point $ (z_1, \dots, z_n) \in M_{n+1}$, we can consider the ratio $ \frac{ z_i - z_k }{z_i - z_j}$.  This defines a function $  \mu_{ijk} : M_{n+1} \rightarrow \Cx$.  These extend to well-defined functions $ \mu_{ijk} : \overline M_{n+1} \rightarrow \BP^1$.  These functions satisfy some standard relations and determine an embedding of $ \overline M_{n+1}$ into a product of projective lines (see \cite[\S 4.1]{IKLPR} for more details).

Let $ F_n = \C^n \setminus \Delta / \C $, the set of $ n $ distinct points on an affine line modulo simultaneous translation.  In \cite[\S 6]{IKLPR}, we constructed a compactification of this space, which we call the cactus flower moduli space $\overline F_n $.  For each $ i < j $, the space $ F_n $ has a function $ \delta_{ij} : F_n \rightarrow \Cx $ defined by $ \delta_{ij}(z) = z_i - z_j $.  By the construction of $ \overline F_n $, these functions extend to $ \delta_{ij} : \overline F_n \rightarrow \BP^1 $.

These functions $ \delta_{ij}$ are not enough to specify a point of $ \overline F_n $.  For this we need the $ \mu$ coordinates as well.  More precisely, the space $ \overline F_n $ is covered by affine open sets $ \tilde U_\CS $ indexed by set partitions on $ \CS $ of $ \{1, \dots, n\}$.  On each such affine open set $ \tilde U_\CS$, some of the $ \mu_{ijk} $ coordinates are well defined --- those for which $ i,j,k$ all lie in the same part of the set partition $ \CS $ (see \cite[\S 6.1]{IKLPR}).

The definition of $ \overline F_n$ leads to the following result \cite[Prop 6.8]{IKLPR}.
\begin{thm} \label{th:infty}
\begin{enumerate}
\item  The locus $$ \{ z \in \overline F_n : \delta_{ij}(z) = \infty, \text{ for all } i, j \} $$ is a single point, denoted $ \infty \in \overline F_n $ and called the maximal flower point. 
\item The locus $$ \{ z \in \overline F_n : \delta_{ij}(z) = 0, \text{ for all } i, j \} $$ is a divisor, isomorphic to $ \overline M_{n+1} $. 
\end{enumerate}
\end{thm}

The spaces $ \overline M_{n+1} $ and $ \overline F_n $ come equipped with natural actions of $ S_n $ permuting the labels of the marked points (for $ \overline M_{n+1}$, we only permute the first $ n $ marked points).  Note that $ \dim \overline F_n = n -1 $ and $ \dim \overline M_{n+1} = n-2$.


\subsection{Real loci and their fundamental groups}

We will be particularly interested in the real loci $ \overline F_n(\BR), \overline M_{n+1}(\BR) $ of these moduli spaces, and their equivariant fundamental groups.  Recall for any reasonable topological space $ X $ with an action of a finite group $ G $ and a chosen basepoint $ x \in X$, the equivariant fundamental group $ \pi_1^G(X, x) $ is defined as the set of pairs $(g,p)$ where $ g \in G $ and $p $ is a homotopy class of path from $ x$ to $ g x $.  Multiplication in this group is given by $ (g_1, p_1)(g_2, p_2) = (g_1 g_2, p_1 * g_1(p_2))$ where $ * $ denotes concatenation of paths (from left to right).

There is a group homomorphism $ \pi_1^G(X,x) \rightarrow G $ taking $ (g,p) $ to $ g $.  If $ X $ is path connected, we have a short exact sequence
$$
1 \rightarrow \pi_1(X,x) \rightarrow \pi_1^G(X,x) \rightarrow G \rightarrow 1
$$

Suppose that $ Y \subset X $ is stable under the action of $ G $.  Let $ y \in Y $ be a basepoint. Let $ p_y^x $ be a path from $ y $ to $ x$ in $ X$.   Once such a path is fixed, we will adopt the convention that $ p_{x}^{y} $ denotes the same path in the opposite direction.

Then we get a group homomorphism, 
\begin{equation} \label{eq:pi1map} \widehat{p_y^x} : \pi_1^G(Y, y) \rightarrow  \pi_1^G(X, x) \quad (g,p) \mapsto (g, p_x^y * p * g(p_y^x) ) \end{equation}
Note that $ \widehat{p_y^x} $ only depends on the homotopy class (with fixed endpoints) of $ p_y^x$. 

 We now specify a basepoint $ y \in \overline M_{n+1}(\BR) $, as follows. We define $ y = (1, 2,\dots, n) \in \R^n $; we also use $ y $ for its image in $ \overline M_{n+1}(\BR)$.  At this point $ y $, the standard coordinates $ \mu_{ijk} $ are given by $ \mu_{ijk}^0 := \frac{i -k}{i-j} $.

We have the following results; the first is due to Davis-Januszkiewicz-Scott \cite{DJS} and the second we proved in \cite[Thm 11.11]{IKLPR}.
\begin{thm} \label{th:pi1}
\begin{enumerate}
\item 
 There is an isomorphism $ \pi_1^{S_n}(\overline M_{n+1}(\BR), y) \cong C_n $.
    \item 
 There is an isomorphism $ \pi_1^{S_n}(\overline F_n(\BR), \infty) \cong vC_n $.
\end{enumerate}
\end{thm}
 
We now define a path $ p^{y}_{\infty}$ between the basepoint $ \infty \in \overline F_n(\BR)$  and the basepoint $ y \in \overline M_{n+1}(\R)$.  This path has domain $ [-\infty,0] $; this looks a bit strange, but this domain is homeomorphic to $[0,1]$.  We define $ p : [-\infty,0] \rightarrow \overline F_n(\BR)$ by $p(-\infty) = \infty$, $ p(0) = y$, and for $ t\in (-\infty,0)$, we define $ p(t) = (t, 2t, \dots, nt) \in F_n(\BR) $.

In terms of the coordinates $ \mu, \delta $ on $ \overline F_n $, we have $ \delta_{ij}(t) = t (i-j) $ and all $ \mu_{ijk}$ are constant and equal to $ \mu_{ijk}^0$ (where they are defined).  

 In \cite[\S 9]{IKLPR}, we constructed a cube complex $ \widehat D_n $ and a homeomorphism $ \widehat D_n \rightarrow \overline F_n(\BR) $.  This homeomorphism restricts to a homeomorphism $ H_n \rightarrow \overline M_{n+1}(\BR) $ where $H_n $ is a certain hyperplane of $ \widehat D_n $.

Under the homeomorphism between $ H_n $ and $ \Mr{n+1} $, the point $ y$ corresponds to the point in $ H_n $ labelled by the planar tree with one internal vertex and leaves occuring in the order $ 1, \dots, n $ (this is because $ y_1 < \dots < y_n $ and $ y_1 - y_2 = \dots = y_{n-1} - y_n$).  Moreover, under this same homeomorphism, the path $ p^{y}_\infty $ corresponds to traversing the sub-$1$-cube labelled by this same tree in $ \widehat D_n$; for this reason we will also call this tree $ p^{y}_\infty$.

The inclusion $ \overline M_{n+1} \rightarrow \overline F_n $ from Theorem \ref{th:infty}(2) together with the path $ p^{\infty}_{y}  $ defines by (\ref{eq:pi1map}) a group homomorphism  $ \widehat p^\infty_{y} : \pi_1^{S_n}(\overline M_{n+1}(\BR), y) \rightarrow \pi_1^{S_n}(\overline F_{n}(\BR), \infty)$.

\begin{thm} \label{th:pi12}
\begin{enumerate}
\item The $S_n$-fixed point $ \infty \in \overline F_n $ from Theorem \ref{th:infty}(1) induces the natural map $ S_n \rightarrow vC_n $ on equivariant fundamental groups.
\item Under the isomorphisms from Theorem \ref{th:pi1}, $ \widehat{ p_{y}^\infty} $ becomes the natural map $ C_n \rightarrow vC_n $.
\end{enumerate}
\end{thm}

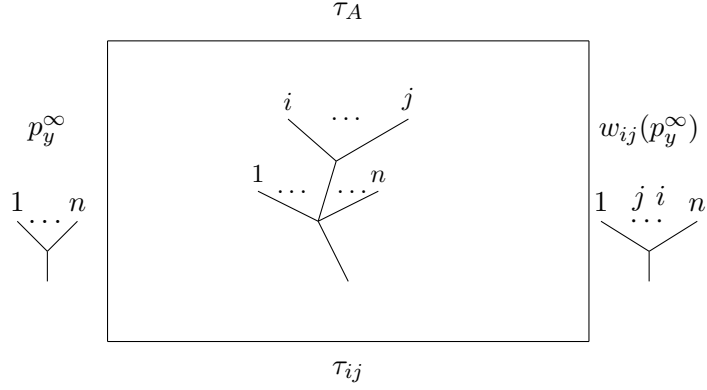
\begin{figure} 
\begin{tikzpicture}[scale=0.8,every node/.style={scale=1}]
	\draw (0,0) -- (8,0);
	\draw (8,0) -- (8,5);
	\draw (0,5) -- (8,5);
	\draw (0,0) -- (0,5);
        \node at (4,-0.5) {$\tau_{ij}$};

        \node at (-1, 3.5) {$p_y^\infty$};

        \draw (-1, 1) -- (-1, 1.5);
        \draw (-1, 1.5) -- (-1.5, 2) node[above] {1};
        \draw (-1, 1.5) -- (-0.5, 2) node[above] {$n$};
        \node at (-1, 2) {$\dots$};

        \node at (9, 3.5) {$ w_{ij}(p_y^\infty)$};
        \draw (9, 1) -- (9, 1.5);
        \draw (9, 1.5) -- (8.2, 2) node[above] {1};
        \draw (9, 1.5) -- (9.8, 2) node[above] {$n$};
        \node at (9, 2) {$\cdots $};
        \node at (9,2) [above] {$ j \ i$};
        
        \node at (4,5.5) {$\tau_A$};

        \begin{scope}[thin,scale=1,every node/.style={scale=0.9}]
	\draw (4,1) -- (3.5,2);
        \draw (3.5,2) -- (2.5,2.5) node[above] {1};
        \draw (3.5,2) -- (4.5,2.5) node[above] {$n$};
        \draw (3.5,2) -- (3.8, 3) ;
        \draw (3.8,3) -- (3, 3.7) node[above] {$i$};
        \draw (3.8,3) -- (5, 3.7) node[above] {$j$};
        \node at (4,3.7) {$\dots$};
        \node at (3.1,2.5) {$\dots$};
        \node at (4.1,2.5) {$\dots$};
        \end{scope}	
\end{tikzpicture} 
\caption{A half 2-cube in $ \widehat D_n$.  The bottom edge of this cube lives in the hyperplane $ H_n$; the labels of its cubes coincide with the labels on the cubes above them.} 
\label{fig:2cube}
\end{figure}

\begin{proof}
    (1) follows immediately from construction of the isomorphism in \cite[Lemma 11.5]{IKLPR}.

    For (2), consider a generator $ s_{ij} $ of $ C_n$.  By the definition of the isomorphism in Theorem \ref{th:pi1}(1), this generator corresponds to the pair $ (w_{ij}, \tau_{ij}) $ where $ \tau_{ij} $ is a $1$-cube inside $ H_n$.  Specifically, $\tau_{ij}$ is the $1$-cube of $H_n$ labelled by the planar tree $\tau_{ij}$ with two internal vertices $ v_1, v_2$, where $ v_1 $ is connected to the leaves $ 1, \dots, i -1$, the vertex $v_2$ and then the leaves $ j+1, \dots, n$, and where $ v_2 $ is connected to leaves $ i, \dots, j$.  If we consider the half 2-cube of $\widehat D_n $ labelled by this same planar tree, we see that its boundary consists of the 1-cubes labelled by $ p_y^\infty, \tau_{ij}, w_{ij}(p_y^\infty), \tau_A $ where $\tau_A $ is the planar forest obtained by deleting the trunk of $ \tau_{ij} $ (here $ A $ is the ordered subset $ (i, \dots, j)$ of $ \{1,\dots, n\}$), as shown in Figure \ref{fig:2cube}.  Hence inside $\pi_1^{S_n}(\widehat D_n)$, we have an equality $ (w_{ij}, p^{rev} * \tau_{ij} * w_{ij}(p)) = (w_{ij}, \tau_A) $. Under the isomorphism $ \pi_1^{S_n}(\widehat D_n) \cong vC_n $ of \cite[Lemma 11.5]{IKLPR}, we see that $ (1, \tau_A)$ is mapped to the generator $ s_A $ of the pure virtual cactus group.  By the definition of $ s_A $ given in \cite[\S 10.7]{IKLPR}, this implies that $ (w_{ij}, \tau_A) $ is mapped to $ s_{ij} \in vC_n$.  To summarize, under the composition
$$C_n \xrightarrow{\cong} \pi_1^{S_n}(\overline M_{n+1}(\BR), y) \xrightarrow{\widehat p_\infty^y} \pi_1^{S_n}(\overline F_{n}(\BR), \infty) \xrightarrow{\cong} vC_n$$
we have
$$ s_{ij} \mapsto (w_{ij}, \tau_{ij}) \mapsto (w_{ij}, (p_y^\infty)^{rev} * \tau_{ij} * w_{ij}(p_y^\infty)) = (w_{ij}, \tau_A) \mapsto s_A w_{ij} = s_{ij} $$
Thus the generators of $ C_n $ are sent to the same-named generators of $vC_n$ as desired.
\end{proof}

\subsection{Deformation of the cactus flower space}

In \cite[\S 6]{IKLPR}, we also defined a deformation $ \overline \CF_n \rightarrow \BA^1 $ of $ \overline F_n $.  For any $ \varepsilon \in \BA^1$, we write $ \overline \CF_n(\varepsilon) $ for the fibre over $ \varepsilon$.  By construction $ \overline \CF_n (0) = \overline F_n $ and for $ \varepsilon \ne 0 $, we have $ \overline \CF_n(\varepsilon) \cong \overline M_{n+2} $, which we regard as Deligne-Mumford space where the marked points are labelled by $ 0, \dots, n+1$.

The open subset $ \CF_n \subset \overline \CF_n $ can be described as follows.  We have a group scheme $ \BG \rightarrow \BA^1 $ defined as $ \BG = \{ (x; \varepsilon) \in \C^2 : 1 - x \varepsilon \ne 1 \} $.  Then $ \CF_n = \BG^n \setminus \Delta / \BG $
where $$ \BG^n \setminus \Delta = \bigl\{(x_1, \dots, x_n; \varepsilon) \in \C^n \times \C : x_i \ne x_j, 1 - x_i \varepsilon \ne 0  \bigr\} $$

Moreover, in \cite[\S 7]{IKLPR}, we constructed two real forms: a standard real form, which we denote by $ \overline \CF_n^{split}$, and a twisted real form $ \overline \CF_n^{comp}$.  We have $ \overline \CF_n^{split} \rightarrow \R $ and $\overline \CF_n^{comp} \rightarrow i \BR $.  Moreover, we proved the following result \cite[Thm 9.3, 9.4]{IKLPR}.

\begin{thm} \label{th:homotopy}
    The inclusions $ \overline F_n(\BR) \subset \overline \CF_n^{split}$ and $ \overline F_n(\BR) \subset \overline \CF_n^{comp}$ are both homotopy equivalences.  In particular the inclusion maps induce isomorphisms 
    $$ \pi_1^{S_n}(\Fr{n}, \infty) \cong \pi^{S_n}_1(\Fs{n}, \infty) \qquad \pi_1^{S_n}(\Fr{n}, \infty) \cong \pi^{S_n}_1(\Fc{n}, \infty)$$
\end{thm}

The functions $ \delta_{ij} : \overline F_n \rightarrow \BP^1 $ extend to functions $ \delta_{ij} :\overline \CF_n \rightarrow \BP^1$.  On the open locus $ \CF_n = \BG^n \setminus \Delta / \BG $, these functions are given by
$$
\delta_{ij}(x_1, \dots, x_n; \varepsilon) = \frac{ x_i - x_j}{1 - \varepsilon x_j} 
$$
For $ \varepsilon \ne 0 $, the isomorphism $ \overline \CF_n(\varepsilon) \cong \overline M_{n+2}  $ comes from extending the isomorphism $ \CF_n(\varepsilon) \cong  M_{n+2} $ given by $ (x_1, \dots, x_n; \varepsilon) \mapsto (\varepsilon^{-1}, x_1, \dots, x_n)$.  In terms of coordinates this means that $ \varepsilon \delta_{ij} = \mu_{j 0 i} $.

Considering $ \BR$-points for our two real structures, the isomorphism $\overline \CF_n(\varepsilon) \cong \overline M_{n+2}  $ restricts to isomorphisms
\begin{equation} \label{eq:CFepsReal}
\Fs{n}(\varepsilon) \cong \Ms{n+2} \text{ for $\varepsilon \in \BR \setminus \{0\}$} \qquad 
\Fc{n}(\varepsilon) \cong \overline M^{comp}_{n+2} \text{ for $\varepsilon \in i\BR \setminus \{0\}$}
\end{equation}
Here $\overline M^{comp}_{n+2} $ denotes the real form of $ \overline M_{n+2}$ parameterizing stable genus 0 curves, defined over $\R$, carrying one pair of complex conjugate points (labelled by $ 0, n+1$), and $ n $ real points, labelled by $ 1, \dots, n$.

Much as above, we can consider the locus where the functions $\delta$ all take the value $ 0$.  From \cite[Prop 6.12 and Rmk 11.14]{IKLPR}, we see the following deformation of Theorem \ref{th:infty}(2).

\begin{thm} \label{th:deformM}
     The locus $$ \{ z \in \overline \CF_n : \delta_{ij}(z) = 0, \text{ for all } i, j \} $$ is a divisor, isomorphic to $ \overline M_{n+1} \times \BA^1 $ (compatibly with the projections to $ \BA^1$). 
\end{thm}
     This isomorphism is compatible with both real structures, so taking real points, this provides inclusions
     $$  \overline M_{n+1}(\BR) \times \BR \subset \overline \CF_n^{split} \quad \overline M_{n+1}(\BR) \times i\BR \subset \overline \CF_n^{comp} $$
     If we specialize to $ \varepsilon \in \BR \setminus \{ 0\} $, resp. $ \varepsilon \in i\BR \setminus \{ 0\} $, and compose with the isomorphisms from (\ref{eq:CFepsReal}), then we obtain the inclusions 
     $$ \Mr{n+1} \subset \Ms{n+2} \qquad \text{resp. } \Mr{n+1} \subset \overline M^{comp}_{n+2} $$
     The images of these inclusions are the loci of curves of the form $ C_0 \cup C_1 $, where marked points $ z_0, z_{n+1} $ lie on $C_0 \cong \BP^1$ and $ C_1 \in \Mr{n+1}$.

\subsection{Equivariant fundamental group of the split form} \label{se:Splitpi1}
Specializing (\ref{eq:CFepsReal}) at $ \varepsilon = 1$, we get an embedding $ \Ms{n+2} \cong \Fs{n}(1) \subset \Fs{n}$.  We will now study the resulting map on equivariant fundamental groups. 

 We have a basepoint in $ \Ms{n+2}$ defined above, but to avoid confusion with the same named point of $ \overline M_{n+1}(\BR) $, we will write this basepoint as $ \tilde y$; for convenience we fix the representative $ \tilde y = (0, 1 , \dots, n)$.

Let $ S_{n+1}$ be the group of permutations of $ \{0, \dots, n\}$.  We consider $ S_n \subset S_{n+1} $ as the subgroup fixing $ 0 $.  By Theorem \ref{th:pi1}(1), we have the isomorphism $ \pi_1^{S_{n+1}}(\Ms{n+2}, \tilde y) \cong C_{n+1}$.  However, we are interested in the $S_n $-equivariant fundamental group, and so it follows that $ \pi_1^{S_n}(\Ms{n+2}, \tilde y) \cong M C_n $, the \emph{mirabolic Cactus group} which we define as the preimage of $ S_n $ under the natural map $ C_{n+1} \rightarrow S_{n+1} $.

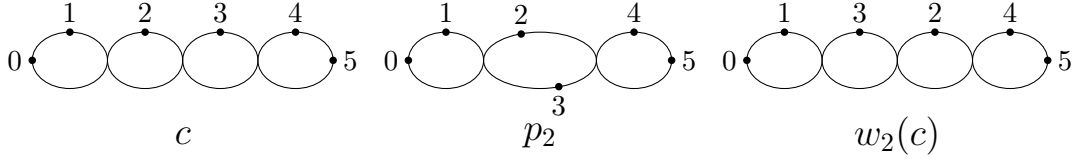
\begin{figure}
    \begin{tikzpicture}[scale=0.25]
        \draw (0,0) ellipse (2 and 1.5);
        \draw (4,0) ellipse (2 and 1.5);
        \draw (8,0) ellipse (2 and 1.5);
        \draw (12,0) ellipse (2 and 1.5);
\node at (-2,0) [circle,fill,inner sep=1pt]{};
\node [left] at (-2,0) {$0$};
\node at (0,1.5) [circle,fill,inner sep=1pt]{};
\node [above] at (0,1.5) {$1$};
\node at (4,1.5) [circle,fill,inner sep=1pt]{};
\node [above] at (4,1.5) {$2$};
\node at (8,1.5) [circle,fill,inner sep=1pt]{};
\node [above] at (8,1.5) {$3$};
\node at (12,1.5) [circle,fill,inner sep=1pt]{};
\node [above] at (12,1.5) {$4$};
\node at (14,0) [circle,fill,inner sep=1pt]{};
\node [right] at (14,0) {$5$};

                \draw (20,0) ellipse (2 and 1.5);
        \draw (25,0) ellipse (3 and 1.5);
        \draw (30,0) ellipse (2 and 1.5);
\node at (18,0) [circle,fill,inner sep=1pt]{};
\node [left] at (18,0) {$0$};
\node at (20,1.5) [circle,fill,inner sep=1pt]{};
\node [above] at (20,1.5) {$1$};
\node at (24,1.4) [circle,fill,inner sep=1pt]{};
\node [above] at (24,1.4) {$2$};
\node at (26,-1.4) [circle,fill,inner sep=1pt]{};
\node [below] at (26,-1.4) {$3$};
\node at (30,1.5) [circle,fill,inner sep=1pt]{};
\node [above] at (30,1.5) {$4$};
\node at (32,0) [circle,fill,inner sep=1pt]{};
\node [right] at (32,0) {$5$};

                \draw (38,0) ellipse (2 and 1.5);
        \draw (42,0) ellipse (2 and 1.5);
        \draw (46,0) ellipse (2 and 1.5);
        \draw (50,0) ellipse (2 and 1.5);
\node at (36,0) [circle,fill,inner sep=1pt]{};
\node [left] at (36,0) {$0$};
\node at (38,1.5) [circle,fill,inner sep=1pt]{};
\node [above] at (38,1.5) {$1$};
\node at (42,1.5) [circle,fill,inner sep=1pt]{};
\node [above] at (42,1.5) {$3$};
\node at (46,1.5) [circle,fill,inner sep=1pt]{};
\node [above] at (46,1.5) {$2$};
\node at (50,1.5) [circle,fill,inner sep=1pt]{};
\node [above] at (50,1.5) {$4$};
\node at (52,0) [circle,fill,inner sep=1pt]{};
\node [right] at (52,0) {$5$};

\node at (6, -4) {\Large $c$};
\node at (25, -4) {\Large $p_2$};
\node at (44, -4) {\Large $w_2(c)$};
\end{tikzpicture}
    \caption{The caterpillar point $ c \in \Ms{4+2}$, a point on the path $ p_2$, and the permuted caterpillar point $ w_2(c) $.} \label{fig:cat}
        
\end{figure}

In \cite{GHR}, we defined elements $ t_i \in C_{n+1} $ for $ i = 1, \dots, n-1$ in the following manner.  Let $ c $ be the standard caterpillar point of $ \overline M_{n+2}$, see Figure \ref{fig:cat}.  Under the inclusion $ \overline M_{n+2} \cong \overline \CF_n(1) \subset \overline \CF_n$, this point lies in the open set $ \CU_{[[n]]}$; on this open set we only have the standard coordinates $ \delta$, which at the point $ c $ are given by
$$\delta_{ij} = \begin{cases} 1 \text{ if $ i < j$} \\ \infty \text{ if $ i > j $}
\end{cases}
$$
For each $ i $, we have the permuted caterpillar point $ w_i(c)$, and a path $ p_i : [0,1] \rightarrow \Ms{n+2} $ given by 
$$\delta_{i \, i+1} = \frac{1}{1-t} \quad \delta_{i+1 \, i} = \frac{1}{t} $$
with all other $ \delta_{jk} $ constant.  This path can also be seen as the image of an interval in $ \Mr{4}$ under the $i$th caterpillar embedding $ \overline M_4 \rightarrow \overline M_{n+2}$ which glues on caterpillars at each end.  

Finally, we need to fix a path $ p_c^{y} $ between the caterpillar point and our standard base point.  This path does not matter precisely, as long as it stays in the open locus 
\begin{equation} \begin{gathered} \label{eq:Un}
    \CU_{[[n]]}(\BR) := \{ z \in \Fs{n} : \delta_{ij}(z) \ne 0 \text{ for all $ i, j$} \} \\
    = \bigl\{ (\nu, \varepsilon) \in \R^{n(n-1)} \times \R : \ \vareps \nu_{ik} + \nu_{ij} \nu_{jk} = \nu_{ik} \nu_{jk} + \nu_{ij} \nu_{ik} \quad \nu_{ij} + \nu_{ji} = \vareps \bigr\}
    \end{gathered}
\end{equation}
Here $ \nu_{ij} = \delta_{ij}^{-1}$.  This open set is a contractible neighbourhood of $ \infty \in \Fs{n}$ (which is the point where $(\nu, \varepsilon) = (0,0)$).

Then we define $ t_i = \widehat p_{y}^c((w_i, p_i))$. 

The caterpillar point $ c$ is invariant under $ w_0 \in S_{n+1}$ and so we can define $ t_0 =\widehat p_{y}^c((w_0, c)) $ where $c $ is the constant path at $ c$ (and as above, $ w_0$ is the transposition $(01)$ and not the long element).

\begin{thm}\cite{GHR} \begin{enumerate}
    \item The group $C_{n+1}$ is generated by $t_0, t_1,\ldots,t_{n-1}$. The standard generators of $C_{n+1}$ are expressed as $s_{0j}=t_0(t_1t_0)(t_2t_1t_0)(t_3t_2t_1t_0)\ldots (t_{j}t_{j-1}\ldots t_1t_0)$.
    \item The subgroup $\wt{C}_n\subset C_{n+1}$ is generated by $t_i$ for $i=1,\ldots, n-1$ and $s_{ij}$ for $1\le i<j\le n$.
\end{enumerate}
\end{thm}

\begin{rem}
    The generators $t_i$ first appeared in the pioneering paper of Berenstein and Kirillov \cite{BK}. The precise relation between the Berenstein-Kirillov group and the cactus group was first established by Chmutov, Glick and Pylyavskyy in \cite{CGP}.
\end{rem}

We now consider the inclusion $ \Mr{n+1} \subset \Ms{n+2}$ as in the discussion after Theorem \ref{th:deformM}.  On this locus $ \mu_{j 0 i} = \infty$ and the inclusion is compatible with all $ \mu_{ijk}$ functions, for $1\le i,j,k \le n$.  

There is a natural path $ p_{\tilde y}^{y}$ with domain $ [-\infty, 0] $ from the basepoint $ \tilde y$ to the basepoint $  y $.  We define this curve by defining  $ p : (-\infty, 0) \rightarrow M_{n+2} $ by $ p(t) = (1, t, 2t, \dots, nt)$.  The limit of this curve at $ t = -\infty$ is the point $ \tilde y$, while its limit at $ t = 0 $ is the point $ y $.  In terms of the $\mu$-coordinates on $ \overline M_{n+2} $, this means that 
$$
\mu_{ijk}(t) = \mu_{ijk}^0  \text{ for $ i,j,k \le n$} \quad \mu_{j 0 i} = \frac{j - i}{j - t^{-1}}   
$$

This path gives us a homomorphism 
$$
\widehat p^{y}_{\tilde y} : C_n = \pi_1^{S_n}(\overline M_{n+1}(\BR), y) \rightarrow \pi_1^{S_n}(\Ms{n+2}, \tilde y)= \tilde C_{n+1}
$$
The following result is immediate from the definitions.

\begin{prop} \label{th:CntCn}
$\widehat p_{y}^{\tilde y}  $ takes all generators $ s_{ij}$ of $ C_n $ to the same named generators of $ \tilde C_{n+1} $.
\end{prop}

Now, we define a path $p_{\infty}^{\tilde y}$ connecting $ \tilde y \in \Ms{n+2} \cong \Fs{n}(1)$ and $ \infty \in \Fr{n} = \Fs{n}(0)$.  For this path, we consider the map $ p : (0,1] \rightarrow \CF_n$ defined by $$ p(s) = (s^{-1} -s , \dots, s^{-1} - ns; s) $$
In terms of the coordinates $ \mu, \delta $, we see that $ \mu_{ijk} = \mu_{ijk}^0 $ is constant and $ \delta_{ij}(s) = \frac{ j - i}{js}$.  From this description, we see that $ \lim_{s\rightarrow 0 } p(s) = \infty$ (since all $ \delta_{ij}$ become $ \infty $), while $ p(1) = (0, -1, \dots, 1-n; 1) $ which under the isomorphism $ \Fs{n}(1) \cong \Ms{n+2}$ becomes $ (1, 0, -1, \dots, 1-n) $ which is equivalent to $ \tilde y $ (this can also be seen easily using the $ \delta $ coordinates).  In fact the path $ p $ is chosen so that for any $ s \in (0,1]$, under the isomorphism $ \Fs{n}(s) \cong \Ms{n+2}$ the point $p(s) $ becomes $ \tilde y$.

Via this map and using Theorem \ref{th:homotopy}, we get a group homomorphism 
$$
\widehat p_{\tilde y}^{\infty} : \tilde C_{n+1} \cong \pi_1^{S_n}(\Ms{n+2}, \tilde y) \rightarrow \pi_1^{S_n}(\Fs{n}, \infty) \cong \pi_1^{S_n}(\Fr{n}, \infty) \cong vC_n
$$

We are now in a position to formulate and prove the following result.

\begin{thm} \label{th:pi1split}
    The group homomorphism $ \widehat p_{\tilde y}^{\infty} : \tilde C_{n+1} \rightarrow vC_n $ is characterized as follows
    \begin{enumerate}
        \item It takes $ C_n \subset \tilde C_{n+1}$ identically to $ C_n \subset vC_n$.
        \item It takes $ t_i \in \tilde C_{n+1}$ to the elementary transposition $ w_i \in S_n \subset vC_n$, for $ i =1, \dots, n-1$.
    \end{enumerate}
\end{thm}

\begin{rem}
    This theorem shows that the homomorphism $ \tilde C_{n+1} \cong \pi_1^{S_n}(\Ms{n+2}, \tilde y) \rightarrow \pi_1^{S_n}(\Fs{n}, \infty) \cong \pi_1^{S_n}(\Fr{n}, \infty) \cong vC_n $ is not faithful, since the generators $t_i$ do not satisfy braid relations in $\wt{C}_n$: indeed, the action of $\wt{C}_n$ on skew-shaped semistandard Young tableaux constructed in \cite{GHR} takes $t_i$ to the $i$-th Bender-Knuth involution which do not satisfy the braid relation.
\end{rem}

\begin{proof}
    We begin with first point.  Let $ D = [-\infty, 0] \times [0,1]$, which is a rectangle with four boundary intervals
\begin{equation} \label{eq:bdy} [-\infty, 0] \times \{0\}, \ \{0\}  \times [0,1], \ [-\infty,0 ] \times \{1\}, \ \{-\infty\}  \times [0,1] \end{equation}

    
    We define $ H : D \rightarrow \Fs{n}$ by first setting
    $$
    H(t,s) = (t, 2t \dots, nt; s)
    $$
    for $ (t,s) \in (-\infty, 0) \times [0, 1]$.  At such a point $ H(t,s) $, the $\mu$ coordinates are constant and equal to $ \mu_{ijk}^0 $ while the $ \delta $ coordinates are given by
    $$
    \delta_{ij}(s,t) = \frac{ j -i}{sj - t^{-1}}
    $$
    From this description, it is clear that $ H$ extends to all of $D $.  Moreover, if we look at the 4 boundary intervals of $ D $ described in (\ref{eq:bdy}), we see that $H $ restricts (respectively) to 
    $$ p_{\infty}^{(y,0)}, \ p_{(y,0)}^{(y,1)}, \ p_{\tilde y}^{(y,1)}, \ p_{\infty}^{\tilde y}$$
    where all our intervals are oriented in the direction of increasing $ t $ or $s $.
    
    For the interval $ \{-\infty \} \times [0,1] $ this is not obvious, but it follows by noting that in $ \BG^n / \BG $, $$(t, 2t \dots, nt; s) = (s^{-1} + t^{-1} - s, s^{-1} + t^{-1} - 2s, \dots, s^{-1} + t^{-1} - ns; s)$$
    and so in the limit $t \to -\infty$ we obtain $ p_{\infty}^{\tilde y}$.


    All of these curves were defined above except for $ p_{(y,0)}^{(y,1)}$ which is simply the obvious path between $(y,0) $ and $(y,1) $ inside $ \Mr{n+1} \times \BR$.  Thus via this homotopy, we conclude that $$ \widehat p_{(y,0)}^\infty \circ \widehat p_{(y,1)}^{(y,0)} =  \widehat p^{\infty}_{\tilde y} \circ \widehat p^{\tilde y}_{(y,1)} $$
    as maps $ C_n \rightarrow vC_n$.  Using Theorem \ref{th:pi12}(2) and Proposition \ref{th:CntCn} and noting that $\widehat p_{(y,1)}^{(y,0)} $ is obviously the identity, this proves (1).

    For the second point, we note that the curves $ p_c^{w(c)}, p_y^c, p_{\tilde y}^{\infty}$ all lie in the $ S_n$ stable open subset $ \CU_{[[n]]}(\BR) $ of $\Fs{n}$.  This open set (\ref{eq:Un}) is contractible and thus any loop made as concatenations of these curves is homotopically trivial.  So we are just left with the equivariant part of this path which is $ w_i$.  
    
    Alternatively, we can explicitly contract the path $ p_i$ to $ \infty$ as follows.  We define $H : [0,1] \times [0,1] \rightarrow \Fs{n}$ by setting
    $$
\varepsilon(s,t) = s \quad \delta_{i \, i+1}(s,t) = \frac{1}{s(1-t)} \quad \delta_{i+1 \, i} = \frac{1}{st}   
    $$
    with all other $ \delta_{ij} $ given by
    $$\delta_{ij} = \begin{cases} s^{-1} \text{ if $ i < j$} \\ \infty \text{ if $ i > j $} \end{cases}
$$
Then $ H$ restricted to $ \{1\} \times [0,1] $ is the path $ p_i$, while $ H $ restricted to $ \{0 \} \times [0,1]$ is the constant path at $ \infty$.  This proves that $ p_i $ is homotopic to the constant path as desired.
\end{proof}

\subsection{Equivariant fundamental group of the compact real form} \label{se:Comppi1}
Now, we consider the compact real form $ \Fc{n}$ which was studied in \cite[\S 11]{IKLPR}.

We will consider the inclusion $$ \Mc{n+2} \cong \Fc{n}(i) \subset \Fc{n} $$
where the first isomorphism comes from specializing (\ref{eq:CFepsReal}) at $\varepsilon = i$.

We choose a basepoint $ u \in \Mc{n+2} $, as follows.  We have $ M_{n+2}^{comp} = U(1)^n \setminus \Delta / U(1)$ and we fix $ u = (\zeta, \zeta^2, \dots, \zeta^n) $ where $ \zeta = e^{2 \pi i / n}$.

The we fix a path $ p_u^\infty $ between $u$ and the basepoint $ \infty$.  This gives us the group homomorphism 
$$ \widehat p_u^\infty : \pi_1^{S_n}(\Mc{n+2}, u) \rightarrow \pi_1^{S_n}(\Fc{n}, \infty) \cong \pi_1^{S_n}(\overline F_n(\BR), \infty) $$
where the last step uses Theorem \ref{th:homotopy}.
 
 We recall the definition of the \emph{extended affine cactus group} $ \widetilde{AC}_n $ by generators and relations from \cite[\S 10.3]{IKLPR}. Consider the set $\mathbb{Z}/n\mathbb{Z} = \{1, \dots, n\} $ with its cyclic order $ 1 < \dots < n < 1 $.  Given an ordered pair $  1 \le i,j \le n $ with $ i \ne j $, we consider the interval $ [i,j] = \{i < i+1 < \dots < j\} $ in this cyclic order. Each interval carries a total order.  We write $ [k,l] \subset [i,j] $ (and say that it is a subinterval) if there is a containment which preserves the orders. The generators of $ \widetilde{AC}_n$ are $s_{ij}$ for all intervals $[i,j]\subset\mathbb{Z}/n\mathbb{Z}$ and another generator $r$, with the following defining relations (all the arithmetic operations are modulo $n$):
 \begin{enumerate}
	 	\item $ s_{ij}^2 = 1 $
	 	\item $ s_{ij} s_{kl} = s_{kl} s_{ij} $ if $ [i,j] \cap [k,l] = \emptyset $
	 	\item $ s_{ij} s_{kl} = s_{i+j-l,i+j-k} s_{ij} $ if $ [k,l] \subset [i,j] $ 
        \item $r^n=1$
        \item $rs_{ij}r^{-1}=s_{i+1,j+1}$.
	 \end{enumerate} 

In particular, $\widetilde{AC}_n$ is generated by $s_{ij}$ with $1\le i<j\le n$ and $r$.  From \cite[Theorem 10.9]{IKLPR}, there is a group homomorphism $\widetilde{AC}_n\to vC_n$ taking $s_{ij}$ to $s_{ij}$, for $ 1 \le i < j \le n$ and $r$ to the long cycle $(1 2 \dots n)$ in $S_n$.      

From Corollary 11.4, Theorems 11.12, and 11.13 of \cite{IKLPR}, we have the following result.

\begin{thm}\label{th:AC} There is an isomorphism $ \pi_1^{S_n}(\Mc{n+2}, u) \cong \widetilde{AC}_n $ which fits into the following diagram  $$
	\begin{tikzcd}
\pi_1^{S_n}(\Mc{n+2}, u) \arrow[d,"\cong"] \arrow[r,"\widehat p_u^\infty"] & \pi_1^{S_n}(\overline F_n(\BR), \infty) \arrow[d,"\cong"]\\
\widetilde{AC}_n  \arrow[r] & vC_n 
	\end{tikzcd}
	$$
\end{thm}

\begin{rem}
From the discussion following Theorem \ref{th:deformM}, we have an inclusion $ \Mr{n+1} \subset \Mc{n+2}$.  Fixing a path between the corresponding basepoints gives a group homomorphism $ C_n \cong \pi_1^{S_n}(\Mr{n+1}, y ) \rightarrow \pi_1^{S_n}(\Mc{n+2}, u) \cong \widetilde{AC}_n$.  It can be shown that this homomorphism is injective and takes generators $s_{ij}$ of $ C_n $ to the same-named generators of $ \widetilde{AC}_n$ (this is the analog of Prop \ref{th:CntCn}).
\end{rem}

\section{Operadic coverings}

\subsection{Operadic structure}

We will be interested in certain maps between these moduli spaces.  Given $ k = 1, \dots, n-1 $, we have the following embeddings (except $ \beta_k $ is also defined when $ k = 0 $):
\begin{equation} \label{eq:op}
\begin{gathered}
 \alpha_k : \overline F_k \times \overline F_{n-k} \rightarrow \overline F_n \quad
 \beta_k : \overline F_{k+1} \times \overline M_{n-k +1} \rightarrow \overline F_n \quad
\gamma_k : \overline M_{k +2} \times \overline M_{n-k +1} \rightarrow \overline M_{n+1} \quad
\end{gathered}
\end{equation}

In terms of curves: $\alpha_k$ attaches two cactus flower curves at their distinguished points, $\beta_k$ attaches the last marked point of a cactus flower curve to the last marked point of a cactus curve, and $ \gamma_k $ attaches the $ (k+1)$st marked point of one cactus curve to the last marked point of another cactus curve.  Of course, we could have defined variants of these maps by changing which marked points we glued.  These different versions are related by the $ S_n $ action.  We have chosen the above conventions so that in each case the maps are $ S_k \times S_{n-k}$-equivariant.

The images of these $\alpha_k$ and $ \beta_k$ are divisors in $\overline F_n$.  Up to the action of $ S_n$, they give all the irreducible components of $ \overline F_n \smallsetminus F_n $.  In the notation of \cite[\S 6.2]{IKLPR}, the image of $\alpha_k $ is the closure of  $ \widetilde V_{\{1, \dots, k \}, \{k+1, \dots, n\}}$ and the image of $\beta_k$ is the closure of $ \widetilde V^{1, \dots, k, \{k+1, \dots, n\}}$.  Note also that $ \beta_0$ provides an inclusion of $ \overline M_{n+1} $ as a divisor in $ \overline F_n $ (as mentioned in Theorem \ref{th:infty}(2)).

Similarly, the image of $\gamma_k$ is a divisor in $ \overline M_{n+1}$ and up to the action of $ S_n$, they give all the irreducible components of $ \overline M_{n+1} \smallsetminus M_{n+1} $.

The maps $\gamma_k$ give $ \{ \overline M_{n+1}  \} $ the structure of an operad, while the maps $ \beta_k$ make $ \{ \overline F_n \} $ into a module over this operad.  As far as we know, the maps $\alpha_k$ have no operadic interpretation.  By a slight abuse of terminology, we will refer to all these maps $\alpha_k, \beta_k, \gamma_k$ as \textbf{operadic maps}.

\subsection{Compatible coverings}\label{se:compcover}
Let $\Xi $ be a finite set.  We can regard $ \overline F_n \times \Xi^n $ and $ \overline M_{n+1} \times \Xi^{n+1} $ as moduli spaces of curves where the marked points are labelled by $ \Xi $.  From this perspective, it is natural to extend our operadic maps (\ref{eq:op}) to 
\begin{equation} \label{eq:op2}
\begin{gathered}
\alpha_k^\Xi :  (\overline F_k \times \Xi^k) \times (\overline F_{n-k} \times \Xi^{n-k}) \rightarrow \overline F_n \times \Xi^n \\
 \beta_k^\Xi : (\overline F_{k+1} \times \Xi^{k+1}) \times_\Xi (\overline M_{n-k +1}  \times \Xi^{n-k+1}) \rightarrow \overline F_n \times \Xi^n \\
\gamma_k^\Xi : (\overline M_{k+2} \times \Xi^{k+2}) \times_\Xi (\overline M_{n-k +1}  \times \Xi^{n-k+1})  \rightarrow \overline M_{n+1} \times \Xi^{n+1}
\end{gathered}
\end{equation}
The definition of $ \alpha_k^\Xi $ is straightforward.  For $ \beta_k^\Xi $, we define  
$$  \Xi^{k+1} \times_\Xi \Xi^{n-k+1} = \{ (\l_1, \dots, \l_{k}, \mu), (\l_{k+1}, \dots, \l_n, \mu)\} $$ 
and then 
$$ \beta_k^\Xi(C_1,(\l_1, \dots, \l_{k}, \mu), C_2,(\l_{k+1}, \dots, \l_n, \mu)) = (\beta_k(C_1, C_2), (\l_1, \dots, \l_n)) $$
while for $ \gamma_k^\Xi $ we define  $$\Xi^{k+2} \times_\Xi \Xi^{n-k+1} = \{(\l_1, \dots, \l_k, \mu, \nu), (\l_{k+1}, \dots, \l_n, \mu) \} $$
and then 
$$ \gamma_k^\Xi(C_1,(\l_1, \dots, \l_{k}, \mu, \nu), C_2,(\l_{k+1}, \dots, \l_n, \mu)) = (\gamma_k(C_1, C_2), (\l_1, \dots, \l_n, \nu)) $$

There are actions of $ S_n $ on $\overline F_n \times \Xi^n $ and $ \overline M_{n+1} \times \Xi^{n+1} $ (once again $ S_n $ permutes the first $ n $ coordinates in $ \Xi^{n+1}$).  The above operadic maps $\alpha_k^\Xi,  \beta_k^\Xi, \gamma_k^\Xi $ are $ S_k \times S_{n-k} $ equivariant.

 A $\Xi$-\textbf{coloured operadic covering of the moduli spaces of cactus flower curves} is a sequence of covering spaces $ \CE_n \rightarrow \overline F_n(\BR) \times \Xi^n $ and $ \CX_n \rightarrow \overline M_{n+1}(\BR) \times \Xi^{n+1} $ along with the following data
 \begin{enumerate}
 \item For each $n $, actions of $ S_n $ on $  \CE_n$, $\CX_{n+1} $, compatible with the actions of $ S_n $ on $\overline F(\BR)$, $\overline M_{n+1}(\BR)$, $\Xi^n $.
\item Compatibility with each of the three types of operadic maps $\alpha_k, \beta_k, \gamma_k$.  More precisely, the data of $S_k \times S_{n-k}$-equivariant isomorphisms
\begin{equation} \label{eq:op3}
\begin{gathered} \tilde \alpha_k :  \CE_k \times \CE_{n-k} \rightarrow  \CE_n \text{ as covers of } \overline F_k(\BR) \times \Xi^k \times \overline F_{n-k}(\BR) \times \Xi^{n-k} \\
\tilde \beta_k : \CE_{k+1} \times_\Xi \CX_{n-k+1} \rightarrow  \CE_n \text{ as covers of }  (\overline F_{k+1}(\BR) \times \Xi^{k+1}) \times_\Xi (\overline M_{n-k +1}(\BR)  \times \Xi^{n-k+1}) \\
\tilde \gamma_k : \CX_{k+2} \times_\Xi \CX_{n-k+1} \rightarrow  \CX_{n+1} \text{ as covers of } (\overline M_{k+2}(\BR) \times \Xi^{k+2}) \times_\Xi (\overline M_{n-k +1}(\BR)  \times \Xi^{n-k+1})
\end{gathered}
\end{equation}
 \end{enumerate}
 Here, in the domain of $ \tilde \beta_k $, we have written $  \CE_{k+1} \times_\Xi \CX_{n-k+1}  $ to mean the preimage of $ \Xi^{k+1} \times_\Xi \Xi^{n-k+1} $ under the map $ \CE_{k+1} \times \CX_{n-k+1} \rightarrow \Xi^{k+1} \times \Xi^{n-k+1}$, and similarly for the domain of $ \tilde \gamma_k $.
  
Note that $ \tilde \beta_0 $ gives an isomorphism between $ \CX_{n+1} $ and the restriction of $\CE_n $ to $ \overline M_{n+1}(\BR) \subset \overline F_n(\BR) $.  Thus $ \CE_n $ almost determines $ \CX_{n+1} $.  There is a bit more data in $\CX_{n+1} $ since it comes with a map to $ \Xi^{n+1}$, whereas $ \CE_n $ only maps to $ \Xi^n $.
 
 Note also that $ \CX_{n+1} $ together with the data of the $ \tilde \gamma_k $ gives a $\Xi$-coloured operadic covering of the moduli spaces $ \overline M_{n+1}(\BR) $ in the sense of \cite[Def. 4.9]{HKRW}.

Given a $\Xi$-coloured operadic covering $ \CE_n, \CX_{n+1}$, we write $ \CE(C, \ul)$ for the fibre of $ \CE_n $ over a point $ (C, \ul) \in \overline F_n(\BR) \times \Xi^n$ and $ \CX(C, \ul, \mu) $ for the fibre of $ \CX_{n+1} $ over a point $ (C, \ul, \mu) \in \overline M_{n+1}(\BR) \times \Xi^{n+1}$.  Note that $ \tilde \alpha_k, \tilde \beta_k, \tilde \gamma_k $ give us (and are determined by) bijections
\begin{equation} \label{eq:operbij}
\begin{gathered}
\tilde \alpha_k : \CE(C_1, (\l_1, \dots, \l_k)) \times \CE(C_2, (\l_{k+1}, \dots, \l_n)) \rightarrow \CE(\alpha_k(C_1, C_2), (\l_1, \dots, \l_n)) \\
\tilde \beta_k : \sqcup_\mu \CE(C_1, (\l_1, \dots, \l_k, \mu)) \times \CX(C_2, (\l_{k+1}, \dots, \l_n, \mu))) \rightarrow \CE(\beta_k(C_1, C_2), (\l_1, \dots, \l_n)) \\
\tilde \gamma_k : \sqcup_\mu \CX(C_1, (\l_1, \dots, \l_k, \mu, \nu)) \times \CX(C_2, (\l_{k+1}, \dots, \l_n, \mu))) \rightarrow \CX(\gamma_k(C_1, C_2), (\l_1, \dots, \l_n, \nu)) \\
\end{gathered}
\end{equation}

There is a category $\XiCov$ of $ \Xi$-coloured operadic coverings. The objects of this category are as above and the morphisms are equivariant maps of covering spaces (trivial on the base) commuting with the $ \tilde \alpha_k, \tilde \beta_k, \tilde \gamma_k $.

\subsection{From concrete coboundary categories to operadic coverings} \label{se:FromCat}
Given a concrete coboundary category, we can use Theorem \ref{th:vCnact} to produce an operadic covering. 

More precisely let $ \cC $ be a $\Xi$-coloured concrete coboundary category.  For each $ n $, let $ \BL_n := \sqcup_{\ul \in \Xi^n} \CL(\l_1) \times \cdots \times \CL(\l_n) $; we have an obvious map $ \BL_n \rightarrow  \Xi^n$ .  By Theorem \ref{th:vCnact}, we have an action of $ vC_n $ on $ \BL_n $, which is compatible with the action of $ S_n $ on $ \Xi^n $.  Thus, via Theorem \ref{th:pi1}, we get an $S_n $-equivariant covering $ \CE_n \rightarrow \overline F_n(\BR) \times \Xi^n $.  On the other hand, just regarding $ \cC $ as a $ \Xi$-coloured coboundary category (forgetting the faithful functor $ \cC \rightarrow \set $), from the construction in section 4.14 of \cite{HKRW}, we get an $ S_n $-equivariant covering $ \CX_{n+1} \rightarrow \overline M_{n+1}(\BR) \times \Xi^{n+1}$.

\begin{prop} \label{pr:buildcover}
This defines a $\Xi$-coloured operadic covering of the moduli spaces of cactus flower curves.
\end{prop}

\begin{proof}
    We must define the maps $ \tilde \alpha_k, \tilde \beta_k, \tilde \gamma_k$.  We will construct these in the form of the bijections (\ref{eq:operbij}).   For $ \tilde \alpha_k $, the bijection 
    $$\CE(C_1, (\l_1, \dots, \l_k)) \times \CE(C_2, (\l_{k+1}, \dots, \l_n)) \rightarrow \CE(\alpha_k(C_1, C_2), (\l_1, \dots, \l_n))$$ 
    comes from the obvious map $ \BL_k \times \BL_{n-k} \rightarrow \BL_n $.  For $ \tilde \beta_k$, the bijection $$\sqcup_\mu \CE(C_1, (\l_1, \dots, \l_k, \mu)) \times \CX(C_2, (\l_{k+1}, \dots, \l_n, \mu))) \rightarrow \CE(\beta_k(C_1, C_2), (\l_1, \dots, \l_n))$$ comes from the bijection (provided the monoidal structure on the functor $ \cC \rightarrow \set$)
    $$
 \sqcup_\mu \CL(\l_1) \times \cdots \times \CL(\l_k) \times \CL(\mu) \times (L(\l_{k+1}) \otimes \cdots \otimes L(\l_n))_\mu \rightarrow  \CL(\l_1) \times \dots \times \CL(\l_n)
    $$
    Finally, for $ \gamma_k $, the bijection  $$ \sqcup_\mu \CX(C_1, (\l_1, \dots, \l_k, \mu, \nu)) \times \CX(C_2, (\l_{k+1}, \dots, \l_n, \mu))) \rightarrow \CX(\gamma_k(C_1, C_2), (\l_1, \dots, \l_n, \nu))$$ comes from the bijection (provided by the associator) 
    $$
    \sqcup_\mu (L(\l_1) \otimes \cdots \otimes L(\l_k) \otimes L(\mu))_\nu \times (L(\l_{k+1}) \otimes \cdots \otimes L(\l_n))_\mu \rightarrow (L(\l_1) \otimes \cdots \otimes L(\l_n))_\nu
    $$
\end{proof}

The construction of this section is functorial and thus defines a functor $ \XiCCob \rightarrow \XiCov$.

\subsection{From operadic coverings to concrete coboundary categories} \label{se:FromOpCover}
Conversely, let $ \CE_n \rightarrow \overline F_n(\BR) \times \Xi^n $, $ \CX_{n+1} \rightarrow \overline M_{n+1}(\BR) \times \Xi^{n+1} $ be a $ \Xi $-coloured operadic colouring.  We will use this to define a concrete coboundary category $ \Phi : \cC \rightarrow \set$ as follows.

First, by \cite[Thm. 4.13]{HKRW}, the data $ \CX_{n+1} $ defines a coboundary category $ \cC $ whose underlying category is $ \Xiset$.  In particular the multiplicity sets in this category are defined as the fibres of $ \CX_3$; more precisely $ (L(\l_1) \otimes L(\l_2))_\mu := \CX(y, ( \l_1, \l_2, \mu)) $ where $ y $ is the unique point of $ \overline M_3$. The associator $ \alpha $ comes from parallel transport in the cover $ \CX_4 $.

Thus, it remains to construct the functor $ \Phi$.  As explained at the end of section \ref{se:ConcreteCat}, this requires two pieces of data.

First, for $ \l \in \Xi $, we set $ \CL(\l) := \CE(\infty, \l)$ where $ \infty $ is the unique point of $ \overline F_1 $.  Now, for any pair $ \l_1, \l_2 $, we must define a bijection $  \phi_{\l_1, \l_2}: \sqcup_\mu (L(\l_1) \otimes L(\l_2))_\mu \times \CL(\mu) \rightarrow \CL(\l_1) \times \CL(\l_2)  $.  To construct this bijection, we consider the covering $ \CE_2 $ over $ \overline F_2(\BR) \times \Xi^2$.   The space $ \overline F_2 $ is 1-dimensional and $ \alpha_{12} : \overline F_2 \rightarrow \BP^1 $ is an isomorphism.

On one hand, we consider the point $ \beta_0(\infty,y) = \beta_0(\overline F_1 \times \overline M_3) \subset \overline F_2 $ which is $ 0 \in \overline F_2 =  \BP^1$.  We have a bijection $$ \tilde \beta_0 : \sqcup_\mu \CE(\infty, \mu) \times \CX(y, (\l_1, \l_2, \mu)) \rightarrow \CE(\beta_0(\infty,y), (\l_1, \l_2)) $$
which thus gives us a bijection $\sqcup_\mu (L(\l_1) \otimes L(\l_2))_\mu \times \CL(\mu) \rightarrow \CE_2(0, (\l_1, \l_2)) $.

On the other hand, we consider the point $ \alpha_1(\infty,\infty) = \alpha_1(\overline F_1 \times \overline F_1 ) \subset \overline F_2$, which is $ \infty \in \overline F_2 = \BP^1$.    We have a bijection $ \tilde \alpha_1 : \CE(\infty,\l_1) \times \CE(\infty, \l_2) \rightarrow \CE(\alpha_1(\infty,\infty), (\l_1, \l_2)) $, and thus a bijection $ \CL(\l_1) \times \CL(\l_2) \rightarrow \CE_2(\infty, (\l_1, \l_2))$.

Hence we can use parallel transport in the cover $ \CE_2$ (restricted to $ (\l_1, \l_2) \in \Xi^2$) from $ 0$ to $ \infty $ through the positive real numbers in $ \overline F_2(\BR) = \BR \BP^1 $ (this is a special case of our path $ p_y^\infty$ from above) to get our desired bijection
$$
\phi_{\l_1, \l_2} : \sqcup_\mu \CL(\mu) \times (L(\l_1) \otimes L(\l_2))_\mu \rightarrow \CL(\l_1) \times \CL(\l_2) 
$$

\begin{prop}\label{pr:buildcategory} 
This defines a $\Xi$-coloured concrete coboundary category.
\end{prop}

\begin{proof}
Because we already have Theorem 4.13 of \cite{HKRW}, it suffices to check that the diagram (\ref{eq:hex1}) commutes.  Translating this diagram in terms of our sets and bijections, and we reach the following diagram for all $\l, \mu, \nu $:
	\begin{equation}  \label{eq:pent}
		\begin{tikzcd}
		  &\sqcup_{\rho, \gamma} L_{\l \mu}^\gamma \times L_{\gamma \nu}^\rho \times \CL(\rho) \ar[r,"{\alpha_{\l, \mu, \nu}}", outer sep = 2 pt] \ar[dl,"{\phi_{\rho, \nu}}", outer sep = -2 pt] & 
		  \sqcup_{\rho, \tau} L_{\l \tau}^\rho \times L_{\mu \nu}^\tau \times \CL(\rho) \ar[dr,"{\phi_{\l, \tau}}", outer sep = -2 pt]  & 
		  \\		  
		  \sqcup_\gamma L_{\l \mu}^\gamma \times \CL(\gamma) \times \CL(\nu)  \ar[dr,"{\phi_{\l, \mu} \times id}", outer sep = -2 pt]& & & \sqcup_\tau \CL(\l) \times L_{\mu \nu}^\tau \times \CL(\tau)  \ar[dl,"{id \otimes \phi_{\mu, \nu} }", outer sep = -2 pt] \\
				& \CL(\l) \times \CL(\mu) \times \CL(\nu) \ar[r,"id", outer sep = 2 pt] &  \CL(\l) \times \CL(\mu) \times \CL(\nu) &
		\end{tikzcd}
		\end{equation}
		where we abbreviate $ L_{\l \mu}^\gamma := (L(\l)\otimes L(\mu))_\gamma$.  The hexagon has effectively collapsed into a pentagon because of the triviality of the associator in the category of sets.
		
		To check the commutativity of this pentagon, we work inside of $ \overline F_3(\BR) $.  In particular, we can consider the dominant Weyl chamber $ \{ (z_1, z_2, z_3) \in F_3(\BR) : z_1 < z_2 < z_3 \} $ and its closure in $ \overline F_3(\BR) $ which we denote by $  \overline F_3(\BR)_+$.  From \cite[Thm 9.17]{IKLPR}, we have a homeomorphism between $ \overline F_3(\BR)$ and a cube complex $ \widehat D_3$ (which is depicted in \cite[Eg 8.9]{IKLPR}).  The image of  $  \overline F_3(\BR)_+$ under this homeomorphism is a union of two sub-2-cubes (they are the sub-2-cubes indexed by the two trees whose leaves are in the order 1, 2, 3).  From this description, we see that $  \overline F_3(\BR)_+$ is homeomorphic to a disk.  The boundary of this disk lies on 5 strata in the geometric stratification of $ \overline F_3 $.  These 5 boundary components are as follows 
		\begin{enumerate}
		\item $ \beta_0([0,1])$ --- here $ \beta_0 : \overline F_1 \times \overline M_4 \rightarrow \overline F_3 $ and we use the isomorphisms $\overline F_1 = pt, \overline M_4 = \BP^1 $.
		\item $ \beta_1([0,\infty])$ --- here $ \beta_1 : \overline F_2 \times \overline M_3 \rightarrow \overline F_3 $ and we use the isomorphisms $\overline F_2 = \BP^1, \overline M_3 = pt $. (To be precise, we are really using $ w \circ \beta_1 $ for a permutation $ w $, since we want to put marked point $ z_1, z_2 $ on the cactus curve.
		\item $ \alpha_2([0,\infty])$ --- here $ \alpha_2 : \overline F_2 \times \overline F_1 \rightarrow \overline F_3$ and we use the isomorphisms $\overline F_2 = \BP^1, \overline F_1 = pt $.
		\item $\alpha_1([0,\infty])$ --- here $ \alpha_1 : \overline F_1 \times \overline F_2 \rightarrow \overline F_3$ and we use the isomorphisms $\overline F_1 = pt, \overline F_2 = \BP^1 $.
		\item $  \beta_1([0,\infty])$ --- here $ \beta_1 : \overline F_2 \times \overline M_3 \rightarrow \overline F_3 $ and we use the isomorphisms $\overline F_2 = \BP^1, \overline M_3 = pt $. 		
		\end{enumerate}
		Examining the definitions carefully, we see that the parallel transport in $ \CE_3 $ along these 5 segments gives us the 5 bijections in the pentagon (\ref{eq:pent}). Since $  \overline F_3(\BR)_+$ is contractible, the diagram (\ref{eq:pent}) commutes.		
\end{proof}

As before, we can interpret this as a functor $ \XiCov \rightarrow \XiCCob$.

\begin{rem}
    In this section, we used the positive real locus of $ \overline F_n $ for $ n = 2, 3$.  It is an interesting question to determine this positive real locus for any $ n$.  We believe that this positive real locus is an associahedron.
\end{rem}

\subsection{From operadic coverings to concrete coboundary categories and back}

Suppose that we begin with a $\Xi$-coloured operadic covering $ \CE_n, \CX_{n+1} $ and use it to define a concrete coboundary category $ \cC \rightarrow \set$ as in the previous section.  Then we can use $ \cC \rightarrow \set $ to produce a $ \Xi$-coloured operadic covering $ \CE_n', \CX_{n+1}' $ following the procedure in section \ref{se:FromCat}.  We will need the following result.

\begin{prop} \label{th:CatCoverBack}
    There is an isomorphism of $ \Xi$-coloured operadic coverings between $\CE_n, \CX_{n+1} $ and $ \CE'_n, \CX'_{n+1}$.  
\end{prop}

\begin{proof}
    By the construction in section \ref{se:FromOpCover}, we have isomorphisms $ \CE_1 \rightarrow \CE'_1 $ and $ \CE_2 \rightarrow \CE'_2$.  Since both coverings are operadic, this extends (via the $ \widetilde \alpha_k $) to isomorphisms between $ \CE_n $ and $ \CE'_n $ over all 0 and 1 dimensional strata of $ \overline F_n$ (these strata are all described in \cite[\S 6.2]{IKLPR}).  By the theory over covering spaces, the isomorphism extends over all $ \overline F_n $.  A similar argument holds for $ \CX_{n+1}$.
\end{proof}

In the reverse direction, if we begin with a concrete coboundary category $ \cC \rightarrow \set$, then build an operadic covering, and then use the covering to build another concrete coboundary category $ \cD $, then it is clear from the construction that there is a canonical equivalence $ \cC \cong \cD$.  

We can summarize this as follows.
\begin{cor}
    The above functors are mutually inverse and give an equivalence of categories $ \XiCCob \rightarrow \XiCov$.
\end{cor}

\subsection{Generalizations}
In this paper, we are concerned with concrete coboundary categories $ \cC $.  Such a category includes the data of a functor $ \cC \rightarrow \set $ from a coboundary category to a symmetric monoidal category, which is (potentially) incompatible with commutativity constraints.  As we have seen, this data leads to an action of the virtual cactus group (Theorem \ref{th:vCnact}) and thus to coverings of the real cactus flower spaces $ \overline F_n(\BR) $ (Proposition \ref{pr:buildcover}).  One can imagine a number of generalizations of these results.  More precisely, there are three types of monoidal categories with commutativity constraints: braided monoidal, coboundary, and symmetric monoidal.  Given any functor $ \Phi: \cC \rightarrow \cD $ where $ \cC, \cD $ are one of these three types of categories (possibly the same type), we get a such a generalization.  Let us highlight a few of these.
\begin{enumerate}
\item Suppose that $ \cC $ is braided monoidal and $ \cD $ is symmetric monoidal.  This seems fairly natural; for example, we could take $ \cC $ to be the representation category of a quantum group, $ \cD $ to be the category of vector spaces, and $ \Phi $ the usual forgetful functor.  In this case, we get an action of the virtual braid group $ vB_n $ on tensor products $ \Phi(V_1) \otimes \cdots \otimes \Phi(V_n) $ of objects; as in Theorem \ref{th:vCnact}, the braid group acts using the quantum group braiding (R-matrices) and the symmetric group acts by permutations of tensors.  This situation has been studied before in the literature, by Brochier \cite[Thm 2.3]{Br}.  On the other hand, we are not aware of a natural real algebraic variety whose equivariant fundamental group is $ vB_n $ (see however the recent paper \cite{Gav}).

\item Suppose that $ \cC $ and $ \cD $ are both symmetric monoidal.  In this case, we get an action of the virtual symmetric group $ vS_n $ (analogous to Theorem \ref{th:vCnact}), and thus a covering of the real matroid Schubert variety $ \overline \ft_n(\BR) $.  However, we do not know any natural examples of this situation.

\item Suppose that $ \cC$ and $ \cD $ are both coboundary.  This does arise in nature; we could take $ \cC =  \fg\text{-}\mathtt{Crys}$, $ \cD = \fl\text{-}\mathtt{Crys} $ for some Levi subalgebra $ \fl \subset \fg $, and $ \Phi $ to be the forgetful functor.  In this case, we will get the action of a group which is generated by two copies of the cactus group: one copy comes from the commutors in $ \gcrys $ and the other copy from the commutors in $ \cD = \fl\text{-}\mathtt{Crys} $.  The relations between these two cactus groups are similar to the relations in the virtual cactus group (section \ref{se:vCn}) except that we would need to define ``cabling'' maps $ C_{n-(j-i)} \rightarrow C_n $.  We conjecture that the resulting group is the equivariant fundamental group of the real locus of the Mau-Woodward space $Q_n$ (see \cite{MW} for the definition of this space and \cite[\S 6]{IKLPR} for its relation with $ \overline F_n$). 
 This conjecture is further motivated by Conjecture 6.10 from \cite{IKR}.  
\end{enumerate}

\section{Gaudin subalgebras}
\subsection{Inhomogeneous Gaudin subalgebras} 
Here we recall the definition and the main properties of  Gaudin subalgebras. We refer the reader to \cite{HKRW,IKR} for more detailed exposition.

Let $\widehat{\fg}_-= \fg \otimes t^{-1} \C[t^{-1}]$ be the ``negative half'' of the affine Lie algebra $\widehat{\fg}$. The \emph{universal Gaudin subalgebra} $\A\subset U(\widehat{\fg}_-)$ is the maximal commutative subalgebra uniquely determined as the centralizer of the quadratic element $\sum\limits_{a=1}^{\dim \fg}(x_a\otimes t^{-1})(x^a\otimes t^{-1})$, where $\{x_a\}$ and $\{x^a\}$ are dual bases in $\fg$ with respect to a non-degenerate invariant inner product.

For any $w\in\mathbb{C}$, we have the corresponding evaluation homomorphism $\varphi_w:\fg \otimes t^{-1} \C[t^{-1}] \rightarrow \fg$ evaluating $t$ at $w$. Next, we have a Lie algebra map $ \fg \otimes t^{-1} \C[t^{-1}] \rightarrow \fg $ given by extracting the coefficient of $t^{-1} $ (here the codomain $ \fg $ is endowed with the trivial Lie bracket).  This yields an algebra morphism $\varphi_\infty:U(\widehat{\fg}_-) \rightarrow S(\fg) $, called \emph{evaluation at $ \infty$}. Thus given non-zero $ w_1, \dots, w_n \in \C$, we can define a map
$$
\varphi_{w_1,\ldots,w_n,\infty}=\varphi_{w_1}\otimes\ldots\otimes\varphi_{w_n}\otimes\varphi_{\infty}:U(\widehat{\fg}_-) \rightarrow U(\fg) \otimes \cdots \otimes U(\fg) \otimes S(\fg).
$$
Let $\uz= (z_1, \dots, z_n) $ be a collection of pairwise distinct complex numbers.  The algebra $\A(\uz,\infty):=\varphi_{w-z_1,\ldots,w-z_n,\infty}(\A)$ is a commutative subalgebra in the diagonal invariants $(U(\fg)^{\otimes n} \otimes S(\fg))^{\fg}$, which does not depend on the choice of $w\in\BC\backslash\{z_1,\ldots,z_n\}$. Given $\chi\in\fg$ we can define the evaluation at $ \chi $, $S(\fg)\to\BC$ (after using the Killing form to identify $ \fg = \fg^* $).  We define $ \varphi_{w_1,\ldots,w_n, \chi} $ to be the composite map
$$
U(\widehat{\fg}_-) \xrightarrow{\varphi_{w_1,\ldots,w_n,\infty}} U(\fg) \otimes \cdots \otimes U(\fg) \otimes S(\fg) \rightarrow U(\fg)^{\otimes n}
$$
where the second map evaluates the last tensor factor at $\chi$. Following $ \CA $ under the homomorphism $ \varphi_{w - z_1, \dots, w - z_n, \chi} $ defines the \emph{inhomogeneous Gaudin algebra} $ \CA_\chi(\uz) \subset U(\fg)^{\otimes n} $ (as above we can choose any $ w $ distinct from $ z_1, \dots, z_n$).  This is a commutative subalgebra of $ U(\fg)^{\otimes n} $ which is maximal commutative when $ \chi $ is regular and $ z_1, \dots, z_n $ are distinct. Since $\A$ lies in $\fg$-invariants in $U(\widehat{\fg}_-)$, the subalgebra $\A_{\chi}(\uz)$ lies in the diagonal invariants of $\fz_\fg(\chi)$, the centralizer of $\chi$ in $\fg$.

We will also be interested in the subalgebras $\A_{\chi_0}(\uz)$ for non-regular $\chi_0\in\fg$.  These algebras are smaller than $\A_{\chi}(\uz)$ when $\chi$ is regular because all their elements are invariant under the adjoint action of the centralizer subalgebra $\fz_\fg(\chi_0)\subset\fg$.  

\begin{prop}
    For semisimple $\chi_0$, the subalgebra $\A_{\chi_0}(\uz)$ is a maximal commutative subalgebra in $(U(\fg)^{\otimes n})^{\fz_\fg(\chi_0)}$.
\end{prop}

In particular, for $ \chi = 0$, we obtain the \emph{homogeneous} Gaudin algebras $\A_0(\uz)= \CA(\uz)$ which is a maximal commutative subalgebra in $(U(\fg)^{\otimes n})^\fg $.

We will now record some elementary properties of these algebras (we refer the reader to \cite{HKRW} for more details). First, we consider the invariance of these algebras under natural changes of the parameters.  The following result was proven in \cite[Lemma 9.2]{HKRW}.
\begin{lem} \label{le:InvarianceA}
Let $ \chi \in \fg, (z_1, \dots, z_n) \in \C^n $.   For all $ a \in \C, s \in \C^\times $, we have
\begin{align*}
\A_{\chi}(z_1 + a, \dots, z_n + a) &= \CA_\chi(z_1, \dots, z_n) \\
\A_{\chi}(sz_1, \dots, sz_n) &= \CA_{s\chi}(z_1, \dots, z_n)
\end{align*}
\end{lem}
Because of this lemma, the subalgebra $ \CA_\chi(z_1, \dots, z_n) $ only depends on $ (z_1, \dots, z_n) $ as a point of $ F_n = \C^n \setminus \Delta /\C$ and the subalgebra $\CA(z_1, \dots, z_n) $ only depends on $ (z_1, \dots, z_n) $ as a point of $ M_{n+1} = (\C^n \setminus \Delta) / \Cx \ltimes \C$.

Next, we consider what happens to these algebras under the action of the symmetric group $ S_n $ on $ U(\fg)^{\otimes n} $:
\begin{lem} \label{le:ConjugateA}
Let $ \chi \in \fg, (z_1, \dots, z_n) \in \C^n $. For all $ \sigma \in S_n $, we have
$$ \sigma(\CA_\chi(z_1, \dots, z_n)) = \CA_\chi(z_{\sigma(1)}, \dots, z_{\sigma(n)})$$
\end{lem}

Next, we consider the compactifications of these families of subalgebras.  We proved the following results in \cite{Rcactus} and \cite{IKR}.

\begin{thm}
\begin{enumerate}
    \item The map  $ M_{n+1} \rightarrow \{ \text{subalgebras of $ ((U\fg)^{\otimes n})^{\fg} $} \}$ extends to an inclusion $ \overline M_{n+1} \rightarrow \{ \text{subalgebras of $ ((U\fg)^{\otimes n})^{\fg} $} \}$.
    \item  Assume that $ \chi $ is regular. The map $ F_n \rightarrow  \{ \text{subalgebras of $ (U\fg)^{\otimes n} $} \}$ extends to an inclusion $ \overline F_n \rightarrow \{ \text{subalgebras of $ (U\fg)^{\otimes n} $} \}$.
\end{enumerate}
\end{thm}

Thus for any $ C\in \overline F_n $, we have the subalgebra $\A_\chi(C) \subset (U\fg)^{\otimes n} $.  We now describe these algebras on the boundary strata of the moduli spaces.

We begin by recalling the following algebra maps
\begin{gather*}
\Delta^{n-k} : (U \fg)^{\otimes k+1} \rightarrow (U \fg)^{\otimes n} \\ \Delta^{n-k}(x_1, \dots, x_{k+1}) = (x_1, \dots, x_k, x_{k+1}, \dots, x_{k+1}) \text{ for $(x_1, \dots, x_{k+1}) \in \fg^{\oplus k+1}$}
\end{gather*}
This is a special case of the map $ \Delta^\CB $ from \cite{IKR}, where $ \CB = \{ \{1\}, \dots, \{k\}, \{k+1, \dots, n\}\}$.

We also have the map
\begin{gather*}
j_k : (U \fg)^{\otimes n-k} \rightarrow (U \fg)^{\otimes n} \\ j_k(x_{k+1}, \dots, x_n) = (0, \dots, 0, x_{k+1}, \dots, x_n) \text{ for $(x_{k+1}, \dots, x_n) \in \fg^{\oplus n-k}$} 
\end{gather*}
In the notation from \cite{IKR}, this map is denoted $ j_{k+1, \dots, n}^n$.  We also let  $ Z_{n-k} = \Delta^{n-k}(1 \otimes \cdots \otimes 1 \otimes Z(U\fg))$, the diagonal embedding of the centre into the $ k+1, \dots, n $ tensor factors.  With all this notation, we have the following operadic description of the limit inhomogeneous and homogeneous Gaudin algebras.

\begin{thm} \label{th:operalg}
\begin{enumerate} 
    \item For any $ C_1 \in \overline F_k, C_2 \in \overline F_{n-k}$, we have
    $$\CA_\chi(\alpha_k(C_1, C_2)) = \CA_\chi(C_1) \otimes \CA_\chi(C_2) $$
    \item For any $ C_1 \in \overline F_{k+1}, C_2 \in \overline M_{n-k+1}$, we have
    $$ \CA_\chi(\beta_k(C_1, C_2)) = \Delta^{n-k}(\CA_\chi(C_1)) \otimes_{Z_{n-k}}  j_k(\CA_\chi(C_2)) 
    $$
    \item For any $ C_1 \in \overline M_{k+2}, C_2 \in \overline M_{n-k+1}$, we have
    $$ \CA(\gamma_k(C_1, C_2)) =  \Delta^{n-k}(\CA(C_1)) \otimes_{Z_{n-k}} j_k(\CA(C_2))  
    $$
\end{enumerate}
\end{thm}

\subsection{Diagonalization of Gaudin algebras}

Let $ \l_1, \dots, \l_n $ be dominant weights of $ \fg$.  Let $ V(\ul) = V(\l_1) \otimes \cdots \otimes V(\l_n)$ be the corresponding tensor product of irreducible representations.  Also let $ \mu $ be another dominant weight and consider the multiplicity space $ \Hom_\fg(V(\mu), V(\ul))$.  

By construction, for any $ C \in \overline F_n $, we have an action of $ \CA_\chi(C) $ on $ V(\ul) $, and for any $ C \in \overline M_{n+1}$, we have an action of $ \CA(C)$ on $ \Hom_\fg(V(\mu), V(\ul))$.  So it is natural to try to diagonalize the action of these commutative algebras.  First, we have the following results \cite{R4} and \cite[Theorem 8.7]{IKR} which show that the actions are always cyclic.

\begin{thm}\label{th:cyclic-inhomogeneous}
\begin{enumerate}
    \item Assume that $ \chi \in \fh^{reg}$.  For any $ C \in \overline F_n $, the subalgebra $\CA_\chi(C)$ acts on $V(\ul)$ with a cyclic vector.
    \item For any $ C \in \overline M_{n+1} $, the subalgebra $\CA(C)$ acts on  $\Hom_\fg(V(\mu), V(\ul))$ with a cyclic vector.
\end{enumerate}
\end{thm}

Next, we have the following results from \cite{FFRy}. 
\begin{thm}\label{th:ss-inhomogeneous} 
\begin{enumerate}
    \item 
For any $\chi\in \fh^{split}$ and $C\in \overline{F}_n(\mathbb{R})$, $\CA_{\chi}(C)$ acts semisimply on $V(\ul)$.
\item 
For any $C\in \overline{M}_{n+1}(\BR)$, $\CA(C)$ acts semisimply on $\Hom_\fg(V(\mu), V(\ul))$.
\end{enumerate}
\end{thm}

Combining these results together, we see that if $ \chi \in \fh^{split}$  is regular and $ C \in \overline{F}_n(\mathbb{R})$, then $\CA_{\chi}(C)$ acts with simple spectrum on $V(\ul)$; i.e. it decomposes $ V(\ul) $ into a direct sum of 1-dimensional eigenspaces.  Similarly, if $ C \in \overline M_{n+1}(\BR)$, then $ \CA(C) $ acts with simple spectrum on $\Hom_\fg(V(\mu), V(\ul)) $.

\subsection{Trigonometric Gaudin subalgebras}

Similarly to the inhomogeneous Gaudin subalgebras, one can define the following trigonometric ones responsible for higher integrals of the \emph{trigonometric Gaudin model}. We need the following version of their definition, see \cite{IKR} for the details. Fix $\theta\in\fh^*$.  We define $ \psi_\theta : (U \fg^{\otimes n+1})^\fg \rightarrow U \fg^{\otimes n}$ to be the composite homomorphism:
\[
 (U \fg^{\otimes n+1})^\fg\to (U \fg^{\otimes n+1}/U \fg^{\otimes n+1}\Delta(\fn_+))^{\fn_+}\hookrightarrow U(\fb_-)\otimes U(\fg)^{\otimes n}\to U(\fg)^{\otimes n}, 
\]
with the first arrow being just the projection of the diagonal $\fg$-invariants to the quantum Hamiltonian reduction by the diagonal $\fn_+$, the second is taking the unique representative in the universal enveloping algebra of the Lie subalgebra complement to the diagonal $\fn_+$, and the last arrow is the evaluation of the first tensor factor on the character determined by $\theta$.

For a genus 0 nodal curve $ C \in \overline M_{n+2}$ with $ n+2$ marked points $ z_0, \dots, z_{n+1}$. The trigonometric Gaudin subalgebra $\mathcal{A}^{trig}_\theta(C)\subset U(\fg)^{\otimes n}$ is the image of the homogeneous one $ \mathcal{A}(C) \subset (U \fg^{\otimes n+1})^\fg$; i.e. we define $ \mathcal{A}_{\theta}^{trig}(C) = \psi_{\theta}(\CA(C))$.

In \cite[\S 7]{IKR}, we studied the degeneration of trigonometric Gaudin algebras to inhomogeneous ones.

For each $ (z_1, \dots, z_n; \varepsilon) \in \CF_n $, we defined 
\[
\CA_\chi(z_1, \dots, z_n; \varepsilon) = \begin{cases}\mathcal{A}^{trig}_{\varepsilon^{-1}\chi}(0, 1 - \varepsilon z_1, \dots, 1 - \varepsilon z_n) = \mathcal{A}^{trig}_{\varepsilon^{-1}\chi}(\varepsilon^{-1},  z_1, \dots, z_n) \text{ if $ \varepsilon \ne 0 $} \\
\mathcal{A}_{\chi}(z_1, \dots, z_n) \text{ if $ \varepsilon = 0 $}
\end{cases}
\]

The following result is Theorems 7.1 and 7.11 from \cite{IKR}.

\begin{thm} \label{th:familyCF}
Assume that $ \chi $ is regular.  This defines a family of subalgebras of $ U \fg^{\otimes n} $ parametrized by $  \CF_n $.  Moreover this extends to a family of subalgebras $ \CA_\chi^\varepsilon(C)$ parametrized by $ \overline \CF_n$.
\end{thm}

In \cite[Proposition~8.8]{IKR}, we proved the following cyclicity result for the action of these subalgebras.

\begin{thm} \label{th:trig-cyclic}
Fix $\chi$ regular and domianant weights $ \ul, \mu$.  For all but finite many $\varepsilon$, all the trigonometric/inhomogeneous Gaudin subalgebras $\CA_\chi^{\varepsilon}(C)\subset U\fg^{\otimes n}$ (for $ C\in \overline \CF_n$) act with a cyclic vector on the weight space $V(\ul)_\mu$. 
\end{thm}

So once semisimple, such $\CA_\chi^{\varepsilon}(C)$ acts there with a simple spectrum. Next, we have the following statement about the semisimplicity:

\begin{thm}\label{th:trig-semisimple}\cite[Theorems~8.11~and 8.15]{IKR}
The subalgebra $\CA_\chi^{\varepsilon}(C)\subset U\fg^{\otimes n}$ acts semisimply on the weight space $V(\ul)_\mu$ at least in the following two cases:
    \begin{enumerate}
        \item $\chi\in\fh^{split}$ and $C\in \Fs{n}$,
        \item $\chi - \frac{\varepsilon}{2} \mu \in \fh^{split}$ and $C \in \Fc{n}$.
    \end{enumerate}
\end{thm}

For the second statement, we note that $ \varepsilon \in i \R$, so that if $ \varepsilon \ne 0 $, then the condition $\chi - \frac{\varepsilon}{2} \mu \in \fh^{split}$ is equivalent to $ \varepsilon^{-1} \chi - \frac{1}{2} \mu \in \fh^{comp}$, which is the condition that appears in \cite[Thm 8.15]{IKR}.

\subsection{A generalization} \label{se:AGen}
When we work over the compact real form, in order to ensure that the condition $ \chi - \frac{\varepsilon}{2} \mu \in \fh^{split}$ holds, we will need to consider a family of trigonometric Gaudin subalgebras where $ \chi $ depends on $ \varepsilon$.  We will now make this precise.

Let $ \chi \in \fh[\varepsilon] $, a polynomial in $ \varepsilon$ with coefficients from $ \fh $; equivalently $ \chi : \BA^1 \rightarrow \fh$ is a morphism of varieties.  Assume that $ \chi(0) \in \fh^{reg} $ and let $ U = \chi^{-1}(\fh^{reg}) \subset \BA^1$, a Zariski open and hence co-finite subset of $ \BA^1$.  Let $ \overline \CF_n(U) = \varepsilon^{-1}(U)$.  (In other words $  \overline \CF_n(U) = \{ (C; \varepsilon) : \chi(\varepsilon) \in \fh^{reg} \} $.)

The method of proof of Theorem 7.11 from \cite{IKR} generalizes to give the following result.

\begin{thm}
    There exists a family of subalgebras of $ U\fg^{\otimes n} $ parametrized by $\overline \CF_n(U) $ given by $ (C; \varepsilon) \mapsto  \CA_{\chi(\varepsilon)}^{\varepsilon}(C)$.
\end{thm}

To apply this result, fix $ \chi \in \fh^{reg} \cap \fh^{split}$.  For any weight $ \mu$, we define $ \chi(\varepsilon) = \chi + \frac{\varepsilon}{2}\mu$.   This gives us a family of subalgebras $ \CA_{\chi + \frac{\varepsilon}{2}\mu}^\varepsilon(C) $ over $ \overline \CF_n(U) $.  Then from Theorem \ref{th:trig-semisimple}, we deduce the following.

\begin{cor} \label{th:trig-semisimple2}  For any choice $ \l_1, \dots, \l_n $ of dominant weights and for all $C \in \Fc{n}(U) $, the subalgebra $ \CA_{\chi + \frac{\varepsilon}{2}\mu}^\varepsilon(C) $ acts semisimply on $ V(\ul)_\mu$.
\end{cor}

\section{Monodromy results}

\subsection{Compatible coverings from spectra of inhomogeneous Gaudin subalgebras}

Fix a real $\chi\in\fh^{split}$ in the dominant Weyl chamber, i.e. $(\alpha,\chi)>0$ for any positive root $\alpha$.  Let $ \CE_\chi(C,\ul) $ be the set of eigenlines of $ \CA_\chi(C) $ acting on $ V(\ul) $.

The sets $ \CE_\chi(C,\ul) $ form the fibres of a family $ \cE_\chi(\ul) \rightarrow \overline{F}_n(\BR)$, where
$$
\cE_\chi(\ul) = \{ (C, L) : C \in \overline{F}_n(\BR), L \in \PP(V(\ul)) : \text{ $L$ is an eigenline for $ \CA_\chi(C) $} \}
$$
It is easy to see that this is a real algebraic variety. By Theorems \ref{th:cyclic-inhomogeneous} and \ref{th:ss-inhomogeneous}, $ \cE_\chi(\ul) \rightarrow \overline{F}_n(\BR) $ is a covering space.

We let $ \cE_{n} = \cup_{\ul} \cE_\chi(\ul) $, where in the union is taken over $ \ul \in \Lambda_+^{n} $.

Similarly, we consider $ \CA(C) $ acting on $ \Hom_\fg(V(\mu), V(\ul))$ for $ C \in \overline M_{n+1}(\BR)$.  Theorems \ref{th:cyclic-inhomogeneous} and \ref{th:ss-inhomogeneous} imply that $ \CA(C) $ acts with a simple spectrum and so decomposes $  \Hom_\fg(V(\mu), V(\ul)) $ into eigenlines.

Let $ \CX(C,(\ul,\mu)) $ be the set of eigenlines of $ \CA(C) $ acting on $\Hom_\fg(V(\mu), V(\ul)) $.

The sets $ \CX(C, (\ul, \mu)) $ form the fibres of a family $ \CX(\ul, \mu) \rightarrow \overline M_{n+1}(\BR)$, where
$$
\CX(\ul, \mu) = \{ (C, L) : C \in \overline M_{n+1}(\BR), L \in \PP(\Hom_\fg(V(\mu), V(\ul))) : \text{ $L$ is an eigenline for $ \CA(C) $} \}
$$
It is easy to see that this is a real algebraic variety. By Theorems \ref{th:cyclic-inhomogeneous} and \ref{th:ss-inhomogeneous}, $ \CX(\ul, \mu) \rightarrow \overline{M}_{n+1}(\BR) $ is a covering space.

We let $ \CX_{n+1} = \cup_{\ul,\mu} \CX(\ul, \mu) $, where in the union is taken over $ (\ul,\mu) \in \Lambda_+^{n+1} $.

\begin{thm} \label{th:OperadFromEigenlines}
$ \cE_{n}, \CX_{n+1} $ is a $\Lambda_+$-coloured operadic covering of $\overline{F}_n(\BR)$.
\end{thm}

\begin{proof}
First note by Lemma \ref{le:ConjugateA}, we have actions of $ S_n $ on $ \CE_n$ and $ \CX_{n+1}$.

We must construct three sets of bijections.  To begin with, suppose that $ C_1 \in \overline F_k, C_2 \in \overline F_{n-k}$.  Then by Theorem \ref{th:operalg}(1), we have $\CA_\chi(\alpha_k(C_1, C_2)) = \CA_\chi(C_1) \otimes \CA_\chi(C_2)$.  Let $ \ul = (\l_1, \dots, \l_n) \in \Lambda_+^n $ and let $ \ul'= (\l_1, \dots, \l_k) $ and $ \ul'' = (\l_{k+1}, \dots, \l_n)$.  Hence we get a bijection between the eigenlines of $\CA_\chi(\alpha_k(C_1, C_2)) $ and $ \CA_\chi(C_1) \otimes \CA_\chi(C_2) $ acting on $ V(\ul) = V(\ul') \otimes V(\ul'')$, which we write as 
$$
\tilde \alpha_k : \CE_\chi(\alpha_k(C_1, C_2), \ul) \rightarrow \CE_\chi(C_1, \ul') \times \CE_\chi(C_2, \ul'')
$$

Now, suppose that $ C_1 \in \overline F_{k+1} $ and $ C_2 \in \overline M_{n-k+1}$.  Then by Theorem \ref{th:operalg}(2), we have $\CA_\chi(\beta_k(C_1, C_2)) = \Delta^{n-k}(\CA_\chi(C_1)) \otimes_{Z_{n-k}} j_k (\CA(C_2))$.  Let $ \l_1, \dots, \l_n, \mu \in \Lambda_+$.  Hence we get a bijection between the eigenlines of these algebras acting on 
$$ V(\ul) = \oplus_\mu V(\l_1) \otimes \dots \otimes V(\l_k) \otimes V(\mu) \otimes \Hom_\fg(V(\mu), V(\l_{k+1}, \dots, \l_n))$$ 
which we write as
$$
\tilde \beta_k : \CE_\chi(\beta_k(C_1, C_2), (\l_1, \dots, \l_n)) \rightarrow \sqcup_\mu \CE_\chi(C_1, (\l_1, \dots, \l_k, \mu)) \times \CX(C_2, (\l_{k+1}, \dots, \l_n, \mu))) 
$$

Similarly, we construct $$ \tilde \gamma_k : \sqcup_\mu \CX(C_1, (\l_1, \dots, \l_k, \mu, \nu)) \times \CX(C_2, (\l_{k+1}, \dots, \l_n, \mu))) \rightarrow \CX(\gamma_k(C_1, C_2), (\l_1, \dots, \l_n, \nu))$$

Together these bijections gives us the structure of a $ \Lambda_+$-coloured operadic covering.
\end{proof}

By the construction from Section \ref{se:FromOpCover}, the operadic covering $ \CE_\chi, \CX $ defines a concrete coboundary category $ \Phi : \cC \rightarrow \set $.  

An \textbf{equivalence of concrete coboundary categories} is a monoidal functor $ \Psi : \cC \rightarrow \cD $ along with an isomorphism of monoidal functors $ \tau : \Phi^{\cD}\circ \Psi \rightarrow \Phi^C $.  This means that we should have an isomorphism (in $ \set$), $ \tau_A : \Phi^D(\Psi(A)) \rightarrow \Phi^{\cC}(A) $ for each $ A \in \cC $ (natural in $ A$) such that for any $ A, B \in \cC$ diagram
	\begin{equation} \label{eq:equivconcrete}
		\begin{tikzcd}
		  \Phi^{\cD}( \Psi(A \otimes B))  \ar[r,"{\Phi^{\cD}(\psi_{A, B})}"] \ar[d,"{\tau_{A \otimes B}}"] & \Phi_D(\Psi(A) \otimes \Psi(B)) \ar[r,"{\phi^{\cD}_{\Psi(A), \Psi(B)}}"]& \Phi^{\cD}(\Psi(A)) \times \Phi^{\cD}(\Psi(B)) \ar[d,"{\tau_A \times \tau_B}"] \\
				\Phi^{\cC}(A \otimes B)  \ar[rr,"{\phi^{\cC}_{A,B}}"] & & \Phi^{\cC}(A) \times \Phi^{\cC}(B) 
		\end{tikzcd}
		\end{equation}
commutes.

We now give a precise statement of our main result.

\begin{thm} \label{th:Main}
There is an equivalence of concrete coboundary categories $ \Psi : \cC \cong \gcrys $ taking $ L(\l) $ to $ B(\l) $ for each $ \l \in \Lambda_+ $.
\end{thm}

\begin{proof}
In \cite[\S 12.1]{HKRW}, we constructed a $ \fg$-crystal structure on $ \CE_\chi(p,\l)$ and in \cite[Thm 12.3]{HKRW}, we proved that it was isomorphic to $ B(\l) $ as a $ \fg$-crystal.

By \cite[Thm 8.7]{HKRW}, there is an equivalence $ \Psi : \cC \rightarrow \gcrys$ defined by $ \Psi(L(\l)) = \CE_\chi(p, \l) $.  Moreover, the monoidal structure $ \psi $ on this functor was defined in \cite[\S 14.1]{HKRW} using parallel transport.  

Thus, to complete the proof of this theorem, we just need to define the isomorphism $ \tau$ and verify that the diagram (\ref{eq:equivconcrete}) commutes.  By our construction in \S \ref{se:FromOpCover}, for each $ \lambda \in \Lambda_+ $, we have a simple object $ L(\l) \in \cC $.  Moreover, $ \CL(\l) := \Phi^{\cC}(L(\l))$ was defined to be $ \CE_\chi(p,\l)$.  Thus, we define $ \tau_\l : \CE_\chi(p,\l) = \Phi^{\cC}(L(\l)) \rightarrow \Phi^{\gcrys}(\Psi(L(\l)) = \CE_\chi(p,\l) $ to be the identity map.

Examining the diagram (\ref{eq:equivconcrete}) we see that it reduces to showing that for all $ \l_1, \l_2 \in \Lambda_+$, the two maps $ \phi_{\l_1, \l_2} $ and $ \psi_{\l_1, \l_2}$
$$ \sqcup_\mu \CE_\chi(p, \mu) \times \CX(q, (\l_1, \l_2, \mu)) \rightarrow \CE(p, \l_1) \times \CE(p, \l_2))$$
are equal.  Now $ \psi_{\l_1, \l_2} $  is defined in \cite[\S 14.1]{HKRW} using a family of algebras defined in \cite[Prop 10.16]{HKRW}.  But this family is exactly the family of 2-point inhomogeneous Gaudin algebras, as parametrized by $ \overline F_2 $, which gives the cover $ \CE_2$ and therefore is used (according to the construction in \S 4.6) to define $ \phi_{\l_1, \l_2} $.  Hence we see that $ \phi_{\l_1, \l_2} = \psi_{\l_1, \l_2}$ as desired.

\end{proof}



As an application of this theorem, we deduce our monodromy result for inhomogeneous Gaudin algebras.  Let $ \ul \in \Lambda_+^n$.  We have an $S_n$-equivariant cover $ \CE_\chi(\ul) \rightarrow \overline F_n(\BR) $ and thus an action of $ \pi_1^{S_n}(\overline F_n(\BR), \infty) $ on $ \CE_\chi(\infty, \ul) $ (where $ \infty \in \overline F_n $ is the maximal flower point).  By Theorem \ref{th:pi1}, we have an isomorphism $ \pi_1^{S_n}(\overline F_n(\BR), \infty) \cong vC_n$.  On the other hand, we have an action of $ vC_n $ on $ \mathcal B(\ul) $ by Theorem \ref{th:vCnact}.

\begin{cor} \label{th:inmonod}
    There are bijections $ \CE_\chi(\infty, \ul) \cong \mathcal B(\ul)$ for all $ \ul \in \Lambda_+^n$, compatible with the actions of $ vC_n$.
\end{cor}

\begin{proof}
    Let $ \CE'_n$ be the cover produced from $ \gcrys$ by the construction in section \ref{se:FromCat}.  Combining Theorem \ref{th:Main} and Proposition \ref{th:CatCoverBack}, we see that the cover $ \CE'_n $ is isomorphic to our original cover $ \sqcup_{\ul} \CE_\chi(\infty, \ul) $.  Taking monodromy implies the desired result.
\end{proof}

\subsection{Monodromy of trigonometric Gaudin model, split case}
We continue with the same fixed $ \chi \in \fh^{split}$ and fix $ \ul $ a sequence of dominant weights.  Let us write $ \varepsilon : \Fs{n} \rightarrow \R $ for the projection defined by the coordinate $ \varepsilon$.

By Theorem \ref{th:trig-cyclic}, there exists $ c > 0 $, such that $ \CA_\chi(C) $ acts cyclically on $ V(\ul) $ for all $ C \in \Fs{n}(-c, c) := \varepsilon^{-1}(-c, c)$.  And by Theorem \ref{th:trig-semisimple}, it also acts semisimply.

For any $ C \in \Fs{n}(-c, c) $, let $ \CE_\chi^{split}(C, \ul) $ denote the set of eigenlines for the action of $ \CA_\chi^\varepsilon(C)$ on $ V(\ul) $.  This gives an unbranched $S_n$-equivariant covering $ \cup_\ul \CE_\chi^{split}(\ul)$ of $ \Fs{n}(-c, c)$.

\begin{lem}
    Let $ \varepsilon : X \rightarrow \R $ be a variety with a map to the real line.  Assume also that there is an action of $ \Rx $ on $ X $ such that $ \varepsilon$ is $\Rx$-equivariant (with respect to its usual action on $ \R$).  Let $ c> 0 $.  There is a homeomorphism $ \varepsilon^{-1}(-c,c)  \rightarrow X  $ which is the identity on the $ \varepsilon = 0 $ fibre.
\end{lem}
\begin{proof}
    Let $ f : (-c,c) \rightarrow \R $ be the diffeomorphism defined by $ f(t) = \frac{c}{\pi} \tan \frac{ \pi t}{c} $.  Define $ F : \varepsilon^{-1}(-c,c)  \rightarrow X $ by $ F(x) = \frac{f(\varepsilon(x))}{\varepsilon(x)} \cdot x $.  Since $ f $ is a diffeomorphism and $ f'(0) = 1 $, this has the desired properties.
\end{proof}

By \cite[Rem. 7.7]{IKLPR}, this Lemma applies to $ \Fs{n} $ and so we deduce the following.
\begin{cor} \label{co:Expand}
    There is a homeomorphism $ \Fs{n}(-c,c) \cong \Fs{n}$.
\end{cor}

Via this homeomorphism, we will regard $\CE_\chi^{split}(\ul)$ as a covering of $ \Fs{n}$.  Our next monodromy result concerns the monodromy of this covering.  Recall the basepoint $ \tilde y \in \Ms{n+2} = \Fs{n}(1)$ from section \ref{se:Splitpi1}.  Lifting the path $ p^\infty_{\tilde y}$, we get a bijection between the fibre $ \CE_\chi^{split}(\tilde y,\ul)$ and the fibre $ \CE_\chi(\infty, \ul)$, which by Corollary \ref{th:inmonod} we identify with $ \mathcal B(\ul)$.

Recall from section \ref{se:Splitpi1}, the fundamental group $ MC_{n} = \pi_1^{S_n}(\Ms{n+2}, \tilde y)$ has generators $ t_1, \dots, t_{n-1}$ and $ s_{ij} $ for $ 1 \le i < j \le n$.  By restricting the covering $ \CE_\chi^{split}(\ul) $ to $ \Ms{n+2} $ (in other words, to the eigenlines for trigonometric Gaudin algebras), we get a monodromy action of $ MC_n $ on $ \CE_\chi^{split}(\tilde y,\ul)  $.

\begin{thm} \label{th:MonodSplit}
    Under the above bijection $ \CE_\chi^{split}(\tilde y,\ul) = \mathcal B(\ul) $, the monodromy action of $\wt{C}_n$ is given by the following:
    \begin{enumerate}
        \item $t_i$ act as the elementary transposition of factors in $\mathcal B(\lambda_1)\times\ldots\times \mathcal B(\lambda_n)$;
        \item $s_{ij}$ acts as the corresponding generator of $C_n$ on the tensor product of crystals $\mathcal B(\lambda_1)\otimes\ldots\otimes \mathcal B(\lambda_n)=\mathcal B(\lambda_1)\times\ldots\times \mathcal B(\lambda_n)$. 
    \end{enumerate}
\end{thm}

\begin{proof}
Since the cover of $ \Ms{n+2}$ is obtained by restricting the cover of $ \Fs{n}$, and because we used $ p^\infty_{\tilde y} $ to identify the fibres, the action of $  \pi_1^{S_n}(\Ms{n+2}, \tilde y)$ factors through the map $ \widehat p^{\infty}_{\tilde y} $ which was studied in Theorem \ref{th:pi1split}.

Moreover, since $\Fr{n}$ is a deformation retract of $\Fs{n}$, the monodromy of the cover over $ \Fs{n}$ is the same as the monodromy over $ \Fr{n}$.  The result now follows by combining Theorem \ref{th:pi1split} and Corollary \ref{th:inmonod}.
\end{proof}

\subsection{Compact case} 
We continue with same fixed $ \chi $.  Also fix a weight $ \mu$ and a sequence of dominant weights $ \ul$.  

Note that $ \chi + \frac{\varepsilon}{2} \mu $ is regular for all $ \varepsilon \in i \R $.  Hence the family of subalgebras $ \CA^{\varepsilon}_{\chi + \frac{\varepsilon}{2} \mu}(C) $ from section \ref{se:AGen} is well-defined for all $ C \in \Fc{n}$.

As in the previous section, by Theorem \ref{th:trig-cyclic} we can find $ c > 0 $ such that $ \CA^{\varepsilon}_{\chi + \frac{\varepsilon}{2} \mu}(C)$ acts cyclically on $ V(\ul)_\mu $ for all $ C \in \Fc{n}(-c,c)$ (where $\Fc{n}(-c,c) = \varepsilon^{-1}(-ci,ci) $ under $ \varepsilon : \Fc{n} \rightarrow i\BR$). By Theorem \ref{th:trig-semisimple2}, it acts semisimply on $ V(\ul)_\mu$.

For $ C \in \Fc{n}(-c,c)$, we define $ \CE^{comp}_\chi(C, \ul)_\mu $ to be the set of eigenlines for the action of $ \CA^{\varepsilon}_{\chi + \frac{\varepsilon}{2} \mu}(C)$ on $ V(\ul)_\mu$.  As before this gives us an unbranched $ S_n$-equivariant covering $ \CE_\chi(\ul)_\mu $ of $ \Fc{n}(-c,c)$.  Then we define $ \CE_\chi(\ul) = \sqcup \CE_\chi(\ul)_\mu$.

As in Corollary \ref{co:Expand}, we have a homeomorphism $\Fc{n}(-c,c) \cong \Fc{n}$.  So we will regard $ \CE^{comp}_\chi(\ul) $ as a covering of $ \Fc{n}$.  We can restrict this covering to $ \varepsilon = i $, which gives us a covering over $ \Mc{n+2}$ and thus an action of $ \pi_1^{S_n}(\Mc{n+2}, u) $ on $ \CE^{comp}_\chi(u,\ul) $ where $u \in \Mc{n+2}$ is our chosen basepoint.  We will now describe this monodromy action, keeping in mind Theorem \ref{th:AC}, which gives us an isomorphism $ \widetilde{AC}_n \cong  \pi_1^{S_n}(\Mc{n+2}, u)$.

Similarly to the split case, we index the fiber of $\CE_\chi^{comp}(\ul)$ over the basepoint $ u $ by the direct product of crystals $\mathcal B(\lambda_1)\times\ldots\times \mathcal B(\lambda_n)$ by using parallel transport along the path $ p_u^\infty $ from section \ref{se:Comppi1}.  As in the proof of Theorem \ref{th:MonodSplit}, we combine Theorem~\ref{th:AC} and Theorem \ref{th:inmonod}, to obtain the following.

\begin{thm} \label{th:MonodComp}
    Under the above bijection $ \CE_\chi^{comp}(u,\ul) = \mathcal B(\ul)$, the monodromy action of $ \widetilde{AC}_n $ is given as follows: 
    \begin{enumerate}
        \item The generator $r$ acts by the cyclic permutation of factors in $\mathcal B(\lambda_1)\times\ldots\times \mathcal B(\lambda_n)$ \\
        \item For $1\le i<j\le n$,
    $s_{ij}$ acts as the corresponding generator of $C_n$ on the tensor product of crystals $ B(\lambda_1)\otimes\ldots\otimes  B(\lambda_n)=\mathcal B(\lambda_1)\times\ldots\times \mathcal B(\lambda_n)$. 
    \end{enumerate}
\end{thm}

\bigskip

\noindent\footnotesize{{\bf Joel Kamnitzer} \\
Department of Mathematics and Statistics, McGill University, Montreal QC, Canada \\
{\tt joel.kamnitzer@mcgill.ca}} \\

\noindent\footnotesize{{\bf Leonid Rybnikov} \\
Department of Mathematics and Statistics,
University of Montreal, Montreal QC, Canada\\
{\tt leonid.rybnikov@umontreal.ca}}

\end{document}